\documentclass[oneside,english]{amsart}
\usepackage[T1]{fontenc}
\usepackage[latin9]{inputenc}
\usepackage{babel}
\usepackage{enumitem}
\usepackage{amstext}
\usepackage{amsthm}
\usepackage{amssymb}
\usepackage{geometry}
\usepackage[bookmarks=true,bookmarksnumbered=false,bookmarksopen=false,
 breaklinks=true,pdfborder={0 0 1},backref=false,colorlinks=false]
 {hyperref}
\usepackage[capitalise]{cleveref}

\makeatletter

\AtBeginDocument{\providecommand\thmref[1]{\cref{thm:#1}}}
\AtBeginDocument{\providecommand\defref[1]{\cref{def:#1}}}
\AtBeginDocument{\providecommand\figref[1]{\cref{fig:#1}}}
\AtBeginDocument{\providecommand\exaref[1]{\cref{exa:#1}}}
\AtBeginDocument{\providecommand\propref[1]{\cref{prop:#1}}}
\AtBeginDocument{\providecommand\lemref[1]{\cref{lem:#1}}}
\AtBeginDocument{\providecommand\corref[1]{\cref{cor:#1}}}
\numberwithin{equation}{section}
\numberwithin{figure}{section}
\theoremstyle{plain}
\newtheorem{thm}{\protect\theoremname}[section]
\theoremstyle{definition}
\newtheorem{defn}[thm]{\protect\definitionname}
\theoremstyle{plain}
\newtheorem{prop}[thm]{\protect\propositionname}
\ifx\proof\undefined
\newenvironment{proof}[1][\protect\proofname]{\par
	\normalfont\topsep6\p@\@plus6\p@\relax
	\trivlist
	\itemindent\parindent
	\item[\hskip\labelsep\scshape #1]\ignorespaces
}{%
	\endtrivlist\@endpefalse
}
\providecommand{\proofname}{Proof}
\fi
\theoremstyle{remark}
\newtheorem{rem}[thm]{\protect\remarkname}
\theoremstyle{definition}
\newtheorem{example}[thm]{\protect\examplename}
\theoremstyle{plain}
\newtheorem{lem}[thm]{\protect\lemmaname}
\theoremstyle{plain}
\newtheorem{cor}[thm]{\protect\corollaryname}

\usepackage{tikz}
\usetikzlibrary{arrows.meta}

\makeatother

\providecommand{\corollaryname}{Corollary}
\providecommand{\definitionname}{Definition}
\providecommand{\examplename}{Example}
\providecommand{\lemmaname}{Lemma}
\providecommand{\propositionname}{Proposition}
\providecommand{\remarkname}{Remark}
\providecommand{\theoremname}{Theorem}

\begin{document}
\title{Ordinal graphs and their $\mathrm{C}^{*}$-algebras}
\author{Benjamin Jones}
\date{January 19, 2025}
\address{School of Mathematical and Statistical Sciences, Arizona State University,
Tempe, Arizona 85287}
\email{brjone16@asu.edu}
\begin{abstract}
We introduce a class of left cancellative categories we call ordinal
graphs for which there is a functor $d:\Lambda\rightarrow\mathrm{Ord}$
by which morphisms of $\Lambda$ factor. We use generators and relations
to study the Cuntz-Krieger algebra $\mathcal{O}\left(\Lambda\right)$
defined by Spielberg. In particular, we construct a $\mathrm{C}^{*}$-correspondence
$X_{\alpha}$ for each $\alpha\in\mathrm{Ord}$ in order to apply
Ery\"uzl\"u and Tomforde's condition (S) and prove a Cuntz-Krieger
uniqueness theorem for ordinal graphs.
\end{abstract}

\maketitle

\section{Introduction}

A directed graph $E$ is a collection $\left(E^{0},E^{1},r,s\right)$
where $E^{0}$, $E^{1}$ are countable, discrete sets and $r,s:E^{1}\rightarrow E^{0}$.
To each directed graph $E$ there is an associated $\mathrm{C}^{*}$-algebra
$C^{*}\left(E\right)$. This algebra is defined using generators and
relations, as in \cite[Chapter 5]{GRAPHALGS}. Namely, $C^{*}\left(E\right)$
is the $\mathrm{C}^{*}$-algebra which is universal for mutually orthogonal
projections $\left\{ P_{v}:v\in E^{0}\right\} $ and partial isometries
$\left\{ S_{e}:e\in E^{1}\right\} $ with mutually orthogonal ranges
such that for $v\in E^{0}$ and $e\in E^{1}$,
\begin{enumerate}
\item $S_{e}^{*}S_{e}=P_{s(e)}$
\item $P_{r(e)}S_{e}S_{e}^{*}=S_{e}S_{e}^{*}$
\item $P_{v}=\sum_{f\in r^{-1}\left(v\right)}S_{f}S_{f}^{*}$ if $0<\left|r^{-1}\left(v\right)\right|<\infty$
\end{enumerate}
Regarding $\mathbb{N}$ as a monoid under addition, one may equivalently
define a graph as a small category $\Lambda$ with a functor $d:\Lambda\rightarrow\mathbb{N}$
with the following factorization property: for every $e\in\Lambda$
and $n\leq d\left(e\right)$ there are unique $f,g\in\Lambda$ with
$d\left(f\right)=n$ and $e=fg$. This enables many generalizations
of the construction of the algebra $C^{*}\left(E\right)$ above. In
\cite{KGRAPH}, Kumjian and Pask replace $\mathbb{N}$ with the monoid
$\mathbb{N}^{k}$ for some $k\in\mathbb{N}$ to define $k$-graphs.
Then a similar set of generators and relations yields an algebra $C^{*}\left(\Lambda\right)$
which recovers $C^{*}\left(E\right)$ in the case $k=1$.

More recently, Brown and Yetter in \cite{CONDUCHEFIBRATIONS} replace
$\mathbb{N}$ with a category $\mathcal{B}$ and use groupoids to
derive conditions for Cuntz-Krieger uniqueness. In particular, they
obtain Cuntz-Krieger uniqueness when the factorization satisfies certain
conditions and $\mathcal{B}$ is right-cancellative. Lydia de Wolf
further develops this theory and gives examples of this construction
when $\mathcal{B}=\mathbb{N}_{0}^{\mathbb{N}}$ and $\mathcal{B}=\mathbb{Q}_{\geq0}$
in \cite{LYDIADEWOLF}. One may also consider $C^{*}\left(\Lambda\right)$
for a category $\Lambda$ when $\Lambda$ has no such functor $d$
\cite{LCSC}. All that is required in this construction is a small
category $\Lambda$ which is left-cancellative, yet in the case that
$\Lambda$ is a directed graph or $k$-graph, the usual $\mathrm{C}^{*}$-algebra
is recovered.

In this paper, we consider the case in which there is a functor $d:\Lambda\rightarrow\mathrm{Ord}$
with the factorization property, where $\mathrm{Ord}$ denotes the
ordinals. We regard the ordinals as a monoid under ordinal addition
and call the category $\Lambda$ with this functor an ordinal graph.
The natural numbers are a submonoid of $\mathrm{Ord}$, which makes
ordinal graphs a generalization of directed graphs. We use generators
and relations from \cite{LCSC} to analyze $C^{*}\left(\Lambda\right)=\mathcal{O}\left(\Lambda\right)$.
While addition on $\mathbb{N}$ is abelian and cancellative, addition
on $\mathrm{Ord}$ is neither abelian nor right-cancellative, and
this has consequences for the types of relations we may represent
with ordinal graphs. Instead of groupoids, we adopt an approach using
$\mathrm{C}^{*}$-correpsondences and apply results from \cite{CKU4CPA}
to recover a Cuntz-Krieger uniqueness result.

\section{Preliminaries}

If $\left(X,\phi\right)$ is a $\mathrm{C}^{*}$-correspondence over
$A$, we write $a\cdot x$ for $\phi\left(a\right)x$. For each such
correspondence there is a closed ideal $J_{X}\vartriangleleft A$
with $J_{X}=\left(\ker\phi\right)^{\perp}\cap\phi^{-1}\left(\mathcal{K}\left(X\right)\right)$
and an associated Cuntz-Pimsner algebra $\mathcal{O}\left(X\right)$
obtained as the universal $\mathrm{C}^{*}$-algebra for covariant
representations \cite{KATIDEAL}. Maps $\psi:X\rightarrow\mathcal{C}$
and $\pi:A\rightarrow\mathcal{C}$ into a $\mathrm{C}^{*}$-algebra
$\mathcal{C}$ form a representation $\left(\psi,\pi\right)$ when
$\psi$ is $A$-linear, $\pi$ is a {*}-homomorphism, $\psi\left(x\cdot a\right)=\psi\left(x\right)\pi\left(a\right)$,
$\psi\left(x\right)^{*}\psi\left(y\right)=\pi\left(\left\langle x,y\right\rangle \right)$,
and $\psi\left(a\cdot x\right)=\pi\left(a\right)\psi\left(x\right)$.
Associated to each representation $\left(\psi,\pi\right)$ is a {*}-homomorphism
$\left(\psi,\pi\right)^{\left(1\right)}:\mathcal{K}\left(X\right)\rightarrow\mathcal{C}$
where $\left(\psi,\pi\right)^{\left(1\right)}\left(\theta_{x,y}\right)=\psi\left(x\right)\psi\left(y\right)^{*}$.
Then the representation $\left(\psi,\pi\right)$ is covariant if for
every $a\in J_{X}$, $\left(\psi,\pi\right)^{\left(1\right)}\left(\phi\left(a\right)\right)=\pi\left(a\right)$.
For each correspondence $\left(X,\phi\right)$ over $A$, one can
construct correspondences $\left(X^{\otimes n},\phi_{n}\right)$ over
$A$, where $X^{\otimes n}$ is the $n$-fold internal tensor product
of $X$ over $A$. Given a representation $\left(\psi,\pi\right)$
of $\left(X,\phi\right)$, there exists representations $\left(\psi^{\otimes n},\pi\right)$
of $\left(X^{\otimes n},\phi_{n}\right)$ \cite[Proposition 1.8]{TOPALG}.

Throughout the paper we will make use of ordinal arithmetic. If $\leq$
is a partial order on a set $A$, we say $\leq$ well-orders $A$
if for all $E\subseteq A$ with $E\not=\emptyset$, there exists $x\in E$
with $x=\min E$. An order isomorphism between two partially ordered
sets is an increasing bijection $f:A\rightarrow B$ with an increasing
inverse. An ordinal is an order isomorphism class of well-ordered
sets. We denote the order isomorphism class of a well-ordered set
$A$ as $\left[A\right]$.

Ordinal addition is defined so that $\left[A\right]+\left[B\right]=\left[A\sqcup B\right]$
where $A\sqcup B$ is well-ordered by a relation $\leq$ which agrees
with the order on $A$ and $B$ such that $a\leq b$ for all $a\in A$
and $b\in B$. Addition induces a well order on $\text{Ord}$, where
for $\alpha,\beta\in\text{Ord}$, $\alpha\leq\beta$ if and only if
there is $\gamma\in\text{Ord}$ such that $\alpha+\gamma=\beta$.
Similarly, ordinal multiplication is defined by $\left[A\right]\cdot\left[B\right]=\left[A\times B\right]$
where $\left(a,b\right)\leq\left(c,d\right)$ if $b<d$, or if $b=d$
and $a\leq c$. For ordinals $\left[A\right]$ and $\left[B\right]$,
exponentiation $\left[A\right]^{\left[B\right]}$ is defined to be
$\left[C\right]$, where
\[
C=\left\{ f\in A^{B}:B\backslash f^{-1}\left(\min A\right)\text{ is finite}\right\} 
\]
Here $A^{B}$ denotes the set of functions from $B$ to $A$. Then
$C$ is well-ordered, so that for distinct $f,g\in C$ and $x=\max\left\{ b\in B:f\left(b\right)\not=g\left(b\right)\right\} $,
$f<g$ if and only if $f\left(x\right)<g\left(x\right)$ \cite[Chapter XIV]{CARDINALORDINAL}.

If $A$ is a finite well-ordered set, we say $\left[A\right]$ is
finite. Every finite well-ordered set is order isomorphic to a bounded
interval in $\mathbb{N}$, and for $n\in\mathbb{N}$, we identify
the ordinal $\left[\left\{ k\in\mathbb{N}:k<n\right\} \right]$ with
$n$. Furthermore, arithmetic for finite ordinals agrees with arithmetic
on $\mathbb{N}$. We define $\omega=\left[\mathbb{N}\right]$, from
which we see $\mathbb{N}$ is order isomorphic to $\left\{ \alpha\in\text{Ord}:\alpha<\omega\right\} =\left[0,\omega\right)$. 

While ordinal arithmetic and arithmetic on $\mathbb{N}$ share many
useful properties, for infinite ordinals $\alpha,\beta\in\mathrm{Ord}$,
it is no longer necessary that $\alpha+\beta=\beta+\alpha$ or $\alpha\cdot\beta=\beta\cdot\alpha$.
However, we still have $\alpha\cdot\left(\beta+\gamma\right)=\alpha\cdot\beta+\alpha\cdot\gamma$.
We also have a division algorithm for $\mathrm{Ord}$: if $\gamma<\alpha\cdot\beta$,
there exist unique $\beta_{1}<\beta$ and $\alpha_{1}<\alpha$ such
that $\gamma=\alpha\cdot\beta_{1}+\alpha_{1}$. Base expansion in
terms of infinite bases makes sense for ordinals by the following
theorem.
\begin{thm}[{\cite[Chapter XIV 19.3]{CARDINALORDINAL}}]
\label{thm:base-expansion}For every $\alpha>0$ and $\beta>1$ there
exists unique $k\in\mathbb{N}$ such that $\alpha$ may be represented
uniquely as
\[
\alpha=\beta^{\alpha_{1}}\cdot\gamma_{1}+\beta^{\alpha_{2}}\cdot\gamma_{2}+\ldots+\beta^{\alpha_{k}}\cdot\gamma_{k}
\]
where $\alpha_{n}\geq\alpha_{n+1}$ and $0<\gamma_{n}<\beta$.
\end{thm}
When $\beta=\omega$, \thmref{base-expansion} gives a representation
of $\alpha$ named the Cantor normal form. Since in that case the
coefficients $\gamma_{n}$ are finite, every ordinal $\alpha$ may
be written as a finite sum of powers of $\omega$.

For ordinals $\alpha,\beta$ with $\alpha<\beta$, $\omega^{\alpha}+\omega^{\beta}=\omega^{\beta}$.
Applying this fact and the Cantor normal form, it's possible to compute
arbitrary sums and products of ordinals using the algebraic properties
of ordinals we have stated here. Moreover, addition of ordinals is
left cancellative: for $\alpha,\beta,\gamma\in\mathrm{Ord}$, $\alpha+\beta=\alpha+\gamma$
implies $\beta=\gamma$. Thus we have one-sided subtraction of ordinals.
If $\alpha\leq\beta$, we write $-\alpha+\beta$ for the unique ordinal
$\gamma$ such that $\alpha+\gamma=\beta$. Additionally, we write
the expression $\delta-\alpha+\beta$ for the ordinal $\delta+\left(-\alpha+\beta\right)$.
Using this notation, we have $-\omega^{\alpha}+\omega^{\beta}=\omega^{\beta}$
when $\alpha<\beta$, and for $\alpha=0,\beta=1$, this gives $1+\omega=-1+\omega=\omega$.

\section{Ordinal Graphs}

We begin by defining ordinal graphs. The definition is almost identical
to \cite[Definition 1.1]{KGRAPH}, except the range of $d$ is $\mathrm{Ord}$
instead of $\mathbb{N}^{k}$.
\begin{defn}
\label{def:An-ordinal-graph}An \emph{ordinal graph} is a pair $\left(\Lambda,d\right)$
where $\Lambda$ is a small category and $d:\Lambda\rightarrow\mathrm{Ord}$
is a functor such that for every $e\in\Lambda$ and $\alpha\leq d(e)$,
there exist unique $f,g\in\Lambda$ with $d(f)=\alpha$ and $e=fg$.
\end{defn}
Here small category means $\Lambda$ is a \emph{set} of morphisms.
We will often suppress the functor $d$ and simply write $\Lambda$
for an ordinal graph $\left(\Lambda,d\right)$. We call $d$ the length
functor, and elements of $\Lambda$ paths. $\mathrm{Ord}$ is regarded
as a monoid under ordinal addition. Then since $d$ is a functor,
$d\left(ef\right)=d\left(e\right)+d\left(f\right)$ for paths $e,f\in\Lambda$.
We frequently refer to the existence and uniqueness of $f$ and $g$
in \defref{An-ordinal-graph} as unique factorization. 

It may happen that $d$ maps into the ordinals less than $\omega$.
Then since the range of $d$ is order isomorphic to a subset of $\mathbb{N}$,
the factorization property ensures that each path is a product of
paths of length $1$. Hence $\Lambda$ is the category of paths of
a directed graph $E$, where $E^{0}=d^{-1}\left(0\right)$ and $E^{1}=d^{-1}\left(1\right)$.
Since infinite ordinals are not finite sums of the ordinal $1$, we
do not in general have that $\Lambda$ is generated by $d^{-1}\left(1\right)$.
\begin{defn}
Let $\left(\Lambda,d\right)$ be an ordinal graph. We call $v\in\Lambda$
a \emph{vertex} if $d\left(v\right)=0$.
\end{defn}
As a consequence of uniqueness in the factorization property, we obtain
left cancellation.
\begin{prop}
\label{prop:left-cancellative}Every ordinal graph is a left cancellative
small category.
\end{prop}
\begin{proof}
Let $\left(\Lambda,d\right)$ be an ordinal graph, and suppose $e=fg=fh$
for $f,g,h\in\Lambda$. Then $fg$ and $fh$ are factorizations of
$e$, so by unique factorization $g=h$.
\end{proof}
\begin{rem}
It may be tempting to assume that the same argument yields right cancellation.
This is not true, however, because the factorization of $e$ into
$fg$ in \defref{An-ordinal-graph} only allows us to control the
length of $f$. More precisely, $e$ may factor as $fg=hk$ where
$d\left(g\right)=d\left(k\right)$ and $f\not=h$. The following example
illustrates this.
\end{rem}
\begin{example}
Let $\Lambda=\left[0,\omega^{2}\right)$ and $d:\Lambda\rightarrow\mathrm{Ord}$
be the restriction of the identity function. Then $\Lambda$ is a
category with composition given by ordinal addition. This follows
because for every $\alpha,\beta\in\Lambda$, $\alpha<\omega^{2}$
and $\beta<\omega^{2}$, so $\alpha+\beta<\omega^{2}$ and $\Lambda$
is closed under composition. Unique factorization follows from ordinal
subtraction. In particular, for every $\alpha\in\Lambda$ and $\beta\leq d\left(\alpha\right)=\alpha$
there exists unique $\gamma\leq d\left(\alpha\right)$ such that $\beta+\gamma=\alpha$.
Thus $\left(\Lambda,d\right)$ forms an ordinal graph. 

\begin{figure}[h]
\begin{center}
\begin{tikzpicture}[every loop/.style={looseness=40}]
	\draw[->] node[circle, fill, inner sep=0, minimum size=4pt, label=right:$v$] (v) {} edge[in=135, out=45, loop] node[above] {$e$} ();
	\draw[->] (v) edge[in=225, out=315, loop] node[below] {$f: eee\ldots; ef=f$} ();
\end{tikzpicture}
\end{center}

\caption{\label{fig:one-loop-fig}$\Lambda=\left[0,\omega^{2}\right)$ as an
ordinal graph, where $v=0$, $e=1$, and $f=\omega$}
\end{figure}

To distinguish between paths in $\Lambda$ and elements of $\mathrm{Ord}$,
we may define $v=0\in\Lambda$, $e=1\in\Lambda$, and $f=\omega\in\Lambda$.
Recall that each ordinal can be written uniquely as a sum of powers
of $\omega$. Then $vf=ef=f$, and thus every path in $\Lambda$ is
of the form $\underbrace{f\ldots f}_{n\text{ times}}\underbrace{e\ldots e}_{m\text{ times}}v$. 

In \figref{one-loop-fig}, we write $f:eee\ldots$ to denote a path
$f$ which factors as $f=eeeg$ for some path $g$. In general, we
write $h:e_{1}e_{2}e_{3}\ldots$ for a path $h$ which for each $k\in\mathbb{N}$
factors as $h=e_{1}e_{2}\ldots e_{k}g_{k}$ for some paths $g_{k}$
and satisfies
\[
d\left(h\right)=\sup_{n\in\mathbb{N}}\sum_{k=1}^{n}d\left(e_{k}\right)
\]
Indeed, in this example we have $e_{k}=e$ and $h=f$. Since $ef=f$,
the notation is redundant, but it is useful in the next example.
\end{example}
\begin{example}
\label{exa:long-path}Define a category $\Lambda$ generated by $\left\{ v_{n},e_{n},f_{n},g_{n}:n\in\mathbb{N}\right\} $
with relations:
\begin{enumerate}
\item $v_{n}=r\left(e_{n}\right)=r\left(f_{n}\right)=s\left(f_{n}\right)=r\left(g_{n}\right)$
\item $v_{n}=s\left(e_{n-1}\right)$ for $n>0$
\item $g_{n}=f_{n}e_{n}g_{n+1}$
\item $s\left(g_{0}\right)=v_{0}$
\end{enumerate}
Then $\left\{ v_{n}:n\in\mathbb{N}\right\} $ are the identity maps
of $\Lambda$. There is a functor $d:\Lambda\rightarrow\mathrm{Ord}$
with $d\left(v_{n}\right)=0$, $d\left(e_{n}\right)=d\left(f_{n}\right)=1$,
and $d\left(g_{n}\right)=\omega$. We may regard elements of $\Lambda$
as finite compositions of the generators. Then if $e,f,g\in\Lambda$
are distinct with $d\left(g\right)=\omega$, $d\left(e\right)=d\left(f\right)=1$,
and $s\left(e\right)=s\left(f\right)=r\left(g\right)$, we have $eg=fg$
only if $e=f$. This ensures $\Lambda$ has the unique factorization
property.

\begin{figure}[h]
\begin{center}
\begin{tikzpicture}[every loop/.style={looseness=40}]
	\draw (0, 0) node[circle, fill, inner sep=0, minimum size=4pt, label=below:$v_0$] (v0) {};
	\draw (2, 0) node[circle, fill, inner sep=0, minimum size=4pt, label=below:$v_1$] (v1) {};
	\draw (4, 0) node[circle, fill, inner sep=0, minimum size=4pt, label=below:$v_2$] (v2) {};
	\draw (6, 0) node (v3) {\ldots};
	\draw[->] (v0) edge[in=135, out=45, loop] node[above] {$f_0$} ();
	\draw[->] (v1) edge[in=135, out=45, loop] node[above] {$f_1$} ();
	\draw[->] (v2) edge[in=135, out=45, loop] node[above] {$f_2$} ();
	\draw[->] (v3) edge node[above] {$e_2$} (v2) (v2) edge node[above] {$e_1$} (v1) (v1) edge node[above] {$e_0$} (v0);
	\draw[->] (v0) edge[in=225, out=315] node[below] {$g_2: f_2 e_2 f_3 e_3 \dots$} (v2);
\end{tikzpicture}
\end{center}

\caption{The ordinal graph $\Lambda$ in \exaref{long-path}}
\end{figure}
\end{example}
By \propref{left-cancellative} and \cite[10.3]{LCSC}, each ordinal
graph has an associated Toeplitz $\mathrm{C}^{*}$-algebra $\mathcal{T}(\Lambda)$
and a Cuntz-Krieger $\mathrm{C}^{*}$-algebra $\mathcal{O}(\Lambda)$.
We will work with these algebras using generators and relations, which
we record later in this section.
\begin{defn}
If $\Lambda$ is an ordinal graph and $\alpha\in\mathrm{Ord}$, define
the ordinal graph $\left(\Lambda_{\alpha},d_{\alpha}\right)$ by $\Lambda_{\alpha}=\left\{ e\in\Lambda:d(e)<\omega^{\alpha}\right\} $
and $d_{\alpha}=\left.d\right|_{\Lambda_{\alpha}}$. Define $\Lambda^{\alpha}=\left\{ e\in\Lambda:d(e)=\omega^{\alpha}\right\} $.
The elements $v$ of $\Lambda_{0}$ are the \emph{vertices} of $\Lambda$.
\end{defn}
Note that $\Lambda_{0}=\Lambda^{0}$. $\Lambda_{\alpha}$ is well-defined
because whenever $\beta,\gamma<\omega^{\alpha}$, $\beta+\gamma<\omega^{\alpha}$.
From $\Lambda$, we may construct $\mathcal{T}\left(\Lambda_{\alpha}\right)$
and $\mathcal{O}\left(\Lambda_{\alpha}\right)$ for all ordinals $\alpha$.
Since $\Lambda$ is a set and $\mathrm{Ord}$ is a proper class, $\mathcal{T}\left(\Lambda_{\alpha}\right)=\mathcal{T}\left(\Lambda\right)$
and $\mathcal{O}\left(\Lambda_{\alpha}\right)=\mathcal{O}\left(\Lambda\right)$
for sufficiently large $\alpha\in\mathrm{Ord}$. Later we will see
there is much to say about the relationship between these algebras.

Since $\Lambda$ is a category, it comes equipped with source and
range maps $s,r:\Lambda\rightarrow\Lambda$. As with directed graphs,
$s\left(e\right)$ and $r\left(e\right)$ should be interpreted as
the source and range vertices of the path $e$. 
\begin{prop}
$v\in\Lambda$ is a vertex if and only if $v=s\left(e\right)$ for
some $e\in\Lambda$.
\end{prop}
\begin{proof}
Suppose $v=s\left(e\right)$. Then $vv=s\left(e\right)s\left(e\right)=s\left(e\right)=v$,
hence $d\left(v\right)+d\left(v\right)=d\left(v\right)$. Subtracting
$d\left(v\right)$, we see $d\left(v\right)=0$. Since $s\left(r\left(e\right)\right)=r\left(e\right)$,
this argument also implies $r\left(e\right)\in\Lambda_{0}$. On the
other hand, if $d\left(v\right)=0$, then $r\left(v\right)vs\left(v\right)=v$.
Since $d\left(r\left(v\right)v\right)=d\left(r\left(v\right)\right)=0$,
unique factorization implies $v=r\left(v\right)v=r\left(v\right)$,
and $s\left(v\right)=s\left(r\left(v\right)\right)=r\left(v\right)=v$. 
\end{proof}
\begin{defn}
\label{def:factor}For $e\in\Lambda$ with $\alpha\leq d(e)$, write
$e_{\alpha}$ and $e^{\alpha}$ for the unique paths satisfying $e_{\alpha}e^{\alpha}=e$
and $d\left(e_{\alpha}\right)=\alpha$.
\end{defn}
For finite ordinals $\alpha$, $e^{\alpha}$ may be confused with
a power of $e$. We will never use exponents to denote repeated composition.
We must also take care to avoid confusing $e_{\alpha}$ with a member
of a sequence of paths $e_{1},e_{2},e_{3},\ldots$ For finite $\alpha$
and a path $e\in\Lambda$, $e_{\alpha}$ will always refer to the
factorization of $e$. If we wish to denote a sequence of paths $f_{1},f_{2},\ldots$,
then we will avoid ambiguity by not labeling any path in $\Lambda$
as $f$.

We record some calculations that apply to the notation in the previous
definition. We frequently make use of these identities.
\begin{prop}
Let $e,f\in\Lambda$ with $s\left(e\right)=r\left(f\right)$. Then
\begin{enumerate}[label=\alph*.]
\item If $\alpha\leq d\left(e\right)$, $\left(ef\right)_{\alpha}=e_{\alpha}$
and $\left(ef\right)^{\alpha}=e^{\alpha}f$.
\item If $\alpha\geq d\left(e\right)$, then $\left(ef\right)_{\alpha}=ef_{-d\left(e\right)+\alpha}$
and $\left(ef\right)^{\alpha}=f^{-d\left(e\right)+\alpha}$.
\item If $\alpha+\beta\leq d\left(e\right)$, then $\beta\leq d\left(e^{\alpha}\right)$.
Moreover, $\left(e^{\alpha}\right)^{\beta}=e^{\alpha+\beta}$ and
$\left(e^{\alpha}\right)_{\beta}=\left(e_{\alpha+\beta}\right)^{\alpha}$.
\item If $\beta\leq\alpha\leq d\left(e\right)$, $\left(e_{\alpha}\right)^{\beta}=\left(e^{\beta}\right)_{-\beta+\alpha}$
and $\left(e_{\alpha}\right)_{\beta}=e_{\beta}$.
\end{enumerate}
\end{prop}
\begin{proof}
For part a, suppose $\alpha\leq d\left(e\right)$. Note that $d\left(e_{\alpha}\right)=\alpha$
and $e_{\alpha}e^{\alpha}f=ef$, so by unique factorization, $\left(ef\right)_{\alpha}=e_{\alpha}$
and $\left(ef\right)^{\alpha}=e^{\alpha}f$. For b, suppose on the
other hand that $\alpha\geq d\left(e\right)$. Then $d\left(ef_{-d\left(e\right)+\alpha}\right)=d\left(e\right)+d\left(f_{-d\left(e\right)+\alpha}\right)=d\left(e\right)-d\left(e\right)+\alpha=\alpha$,
and $ef=ef_{-d\left(e\right)+\alpha}f^{-d\left(e\right)+\alpha}$.
Unique factorization once again gives the desired equalities. For
part c, assume $\alpha+\beta\leq d\left(e\right)$. Then $\alpha+\beta\leq d\left(e\right)=d\left(e_{\alpha}e^{\alpha}\right)=d\left(e_{\alpha}\right)+d\left(e^{\alpha}\right)=\alpha+d\left(e^{\alpha}\right)$.
Subtracting $\alpha$ from both sides, we see $\beta\leq d\left(e^{\alpha}\right)$.
Also, $e_{\alpha}\left(e^{\alpha}\right)_{\beta}\left(e^{\alpha}\right)^{\beta}=e=e_{\alpha+\beta}e^{\alpha+\beta}$,
and $d\left(e_{\alpha}\left(e^{\alpha}\right)_{\beta}\right)=\alpha+\beta=d\left(e_{\alpha+\beta}\right)$.
By unique factorization, $e_{\alpha}\left(e^{\alpha}\right)_{\beta}=e_{\alpha+\beta}$
and $\left(e^{\alpha}\right)^{\beta}=e^{\alpha+\beta}$. Then since
$d\left(e_{\alpha}\right)=\alpha$, $\left(e_{\alpha+\beta}\right)^{\alpha}=\left(e^{\alpha}\right)_{\beta}$.
Finally, assume for part d that $\alpha\leq\beta\leq d\left(e\right)$.
$\left(e_{\alpha}\right)^{\beta}=\left(e^{\beta}\right)_{-\beta+\alpha}$
follows from part c, and moreover $\left(e_{\alpha}\right)_{\beta}\left(e_{\alpha}\right)^{\beta}e^{\alpha}=e=e_{\beta}e^{\beta}$
with $d\left(\left(e_{\alpha}\right)_{\beta}\right)=\beta=d\left(e_{\beta}\right)$.
Hence $\left(e_{\alpha}\right)_{\beta}=e_{\beta}$ by unique factorization.

If $\Lambda$ is a left cancellative small category which is finitely
aligned in the following sense, then $\mathcal{T}\left(\Lambda\right)$
and $\mathcal{O}\left(\Lambda\right)$ may be understood in terms
of generators, relations, and the exhaustive sets of $\Lambda$. We
record the relevant definitions below.
\end{proof}
\begin{defn}[{\cite[Definition 3.2]{LCSC}}]
A left cancellative category $\Lambda$ is \emph{finitely aligned}
if for all $e,f\in\Lambda$ there exists finite $F\subseteq\Lambda$
such that $e\Lambda\cap f\Lambda=\cup_{g\in F}g\Lambda$.
\end{defn}
\begin{defn}[{\cite[Definition 10.13]{LCSC}}]
If $\Lambda$ is a left cancellative category and $v\in\Lambda$,
$F\subseteq v\Lambda$ is \emph{exhaustive} for $v$ if for all $f\in v\Lambda$
there exists $e\in F$ such that $e\Lambda\cap f\Lambda\not=\emptyset$.
\end{defn}
\begin{lem}
\label{lem:finite-align}Every ordinal graph $\Lambda$ is finitely
aligned. Moreover, if $e,f\in\Lambda$, then $e\Lambda\cap f\Lambda\in\left\{ \emptyset,e\Lambda,f\Lambda\right\} $.
\end{lem}
\begin{proof}
Let $e,f\in\Lambda$ be given, and suppose without loss of generality
that $d(e)\leq d(f)$. Choose $p,q\in\Lambda$ such that $f=pq$ with
$d(p)=d(e)$. If $e\Lambda\cap f\Lambda=\emptyset$, then we are done.
Otherwise, let $g\in e\Lambda\cap f\Lambda$ and choose $r,s\in\Lambda$
such that $g=er=ps$. By unique factorization, $p=e$, and therefore
$e\Lambda\cap f\Lambda=e\Lambda\cap eq\Lambda=e\Lambda$.
\end{proof}
To present $\mathcal{T}\left(\Lambda\right)$ and $\mathcal{O}\left(\Lambda\right)$
in terms of generators and relations, we must study the common extensions
in $\Lambda$.
\begin{defn}[{\cite[Definition 2.4]{LCSC}}]
For a left cancellative category $\Lambda$ and $e,f\in\Lambda$,
$e\approx f$ if there is invertible $g\in\Lambda$ with $eg=f$.
\end{defn}
\begin{defn}[{\cite[Definition 3.1]{LCSC}}]
If $F$ is a subset of a left cancellative category $\Lambda$, the
\emph{common extensions }of $F$ are the elements of $\bigcap_{e\in F}e\Lambda$.
If $e$ is a common extension of $F$, $e$ is \emph{minimal} if for
all common extensions $f$ of $F$ with $e\in f\Lambda$, $e\approx f$.
$\bigvee F$ denotes the set of minimal common extensions of $F$.
For $e,f\in\Lambda$, $e\vee f$ denotes $\bigvee\{e,f\}$.
\end{defn}
\begin{lem}
\label{lem:approx-equal}If $\Lambda$ is an ordinal graph and $e,f\in\Lambda$,
$e\approx f$ if and only if $e=f$.
\end{lem}
\begin{proof}
First suppose $e\approx f$. Choose invertible $g\in\Lambda$ with
$eg=f$. Then $d\left(s\left(g\right)\right)=d\left(s\left(g\right)s\left(g\right)\right)=d\left(s\left(g\right)\right)+d\left(s\left(g\right)\right)$,
hence $d\left(s\left(g\right)\right)=0$. Also, $d\left(g^{-1}g\right)=d\left(s\left(g\right)\right)=d\left(g^{-1}\right)+d\left(g\right)=0$,
so $d\left(g\right)=0$. Thus $d\left(eg\right)=d\left(e\right)+d\left(g\right)=d\left(e\right)=d\left(f\right)$,
and by unique factorization of $eg=fs\left(f\right)$, $e=f$. Clearly
if $e=f$, $es\left(e\right)=f$ and $e\approx f$.
\end{proof}
\begin{lem}
\label{lem:join-paths}If $\Lambda$ is an ordinal graph, $e,f\in\Lambda$,
and $e\Lambda\cap f\Lambda\not=\emptyset$, then either $e\in f\Lambda$
or $f\in e\Lambda$. If $e\in f\Lambda$, $e\vee f=\left\{ f\right\} $,
and if $f\in e\Lambda$, $e\vee f=\left\{ e\right\} $.
\end{lem}
\begin{proof}
Let $e,f\in\Lambda$ with $e\Lambda\cap f\Lambda\not=\emptyset$.
By \lemref{finite-align}, either $e\Lambda\cap f\Lambda=e\Lambda$
or $e\Lambda\cap f\Lambda=f\Lambda$. We may assume without loss of
generality that $e\Lambda\cap f\Lambda=e\Lambda$. Hence the common
extensions of $\left\{ e,f\right\} $ are the elements of $e\Lambda$.
If $g$ is a common extension of $\left\{ e,f\right\} $ and $e\in g\Lambda$,
then $g=eh$ for some $h\in\Lambda$, and $e\in eh\Lambda$. Thus
$d\left(e\right)\geq d\left(e\right)+d\left(h\right)$, and $d\left(h\right)=0$.
Therefore $g=e$, and $e$ is minimal. If $g\in e\Lambda$ is minimal,
then since $e$ is a common extension, $e\approx g$, and by \lemref{approx-equal},
$g=e$. We conclude that $e\vee f=\left\{ e\right\} $.
\end{proof}
We apply the following theorems to obtain generators and relations
for $\mathcal{T}\left(\Lambda\right)$ and $\mathcal{O}\left(\Lambda\right)$.
\begin{thm}[{\cite[Theorem 10.15]{LCSC}, \cite[Theorem 9.7]{LCSC}}]
\label{thm:generators}Let $\Lambda$ be a countable finitely-aligned
left cancellative small category. Then $\mathcal{O}\left(\Lambda\right)$
is the $C^{*}$-algebra which is universal for generators $\left\{ T_{e}:e\in\Lambda\right\} $
and relations 1-4. $\mathcal{T}\left(\Lambda\right)$ is the $C^{*}$-algebra
which is universal for generators $\left\{ T_{e}:e\in\Lambda\right\} $
and relations 1-3.
\begin{enumerate}
\item $T_{e}^{*}T_{e}=T_{s(e)}$
\item $T_{e}T_{f}=T_{ef}$ if $s(e)=r(f)$
\item $T_{e}T_{e}^{*}T_{f}T_{f}^{*}=\bigvee_{g\in e\vee f}T_{g}T_{g}^{*}$
\item $T_{v}=\bigvee_{e\in F}T_{e}T_{e}^{*}$ if $F$ is finite and exhaustive
for $v$
\end{enumerate}
\end{thm}
Note that by \lemref{join-paths}, relation 3 for ordinal graphs may
be rewritten as
\begin{enumerate}[label=(3)]
\item $T_{e}T_{e}^{*}T_{f}T_{f}^{*}=\begin{cases}
T_{e}T_{e}^{*} & e\in f\Lambda\\
T_{f}T_{f}^{*} & f\in e\Lambda\\
0 & \text{otherwise}
\end{cases}$
\end{enumerate}
Even when an ordinal graph $\Lambda$ is not countable, we define
$\mathcal{O}\left(\Lambda\right)$ and $\mathcal{T}\left(\Lambda\right)$
using these relations. If $\mathcal{C}$ is a $\mathrm{C}^{*}$-algebra
and $T=\left\{ T_{e}:e\in\Lambda\right\} \subseteq\mathcal{C}$ is
a family of operators satisfying relations 1-4, then we call $T$
a Cuntz-Krieger $\Lambda$-family. If instead $T$ satisfies relations
1-3, we call $T$ a Toeplitz $\Lambda$-family. Every Toeplitz $\Lambda$-family
$T\subseteq\mathcal{C}$ induces a {*}-homomorphism $\pi:\mathcal{T}\left(\Lambda\right)\rightarrow\mathcal{C}$
mapping generators to members of $T$. Likewise, each Cuntz-Krieger
family induces a {*}-homomorphism $\pi:\mathcal{O}\left(\Lambda\right)\rightarrow\mathcal{C}$.

We record a few calculations which will prove useful in the study
of these $\mathrm{C}^{*}$-algebras. 
\begin{lem}
\label{lem:basic-calcs}Let $\Lambda$ be an ordinal graph and $e,f\in\Lambda$.
The following hold in $\mathcal{T}\left(\Lambda\right)$ and $\mathcal{O}\left(\Lambda\right)$.
\begin{enumerate}[label=\alph*.]
\item $T_{e}$ is a partial isometry.
\item If $v\in\Lambda_{0}$, then $T_{v}$ is a projection.
\item If $s(e)\not=r(f)$, then $T_{e}T_{f}=0$.
\item $T_{e}^{*}T_{f}=\begin{cases}
T_{g} & f=eg\\
T_{g}^{*} & e=fg\\
0 & e\Lambda\cap f\Lambda=\emptyset
\end{cases}$
\end{enumerate}
\end{lem}
\begin{proof}
For part a, notice that $s(e)=s(s(e))$, so by relation 2, $T_{e}T_{s(e)}=T_{es(e)}=T_{e}$.
Then by relation 1, $T_{e}T_{e}^{*}T_{e}=T_{e}T_{s(e)}=T_{e}$, and
thus $T_{e}$ is a partial isometry. For part b, we have $s(v)=v$,
and relation 1 implies $T_{v}^{*}T_{v}=T_{v}$. To prove part c, suppose
$s(e)\not=r(f)$. Then $s(e)\Lambda\cap r(f)\Lambda=\emptyset$, and
by relation 3, $T_{s(e)}T_{s(e)}^{*}T_{r(f)}T_{r(f)}^{*}=0$. Since
$T_{s(e)}$ and $T_{r(f)}$ are projections, this is equivalent to
$T_{s(e)}T_{r(f)}=0$. Now relation 2 implies $T_{e}T_{f}=T_{e}T_{s(e)}T_{r(f)}T_{f}=0$.
Finally, for d notice that if $f=eg$, $d(f)\geq d(g)$. By taking
adjoints if needed, we may assume without loss of generality that
$d(f)\geq d(g)$. Then either $f\in e\Lambda$ or $e\Lambda\cap f\Lambda=\emptyset$.
First suppose $f=eg$, and note $T_{e}^{*}T_{f}=T_{e}^{*}T_{e}T_{g}=T_{s(e)}T_{g}$,
which by relation 2 implies $T_{e}^{*}T_{f}=T_{g}$. Now assume $e\Lambda\cap f\Lambda=0$.
Then by relation 3, $T_{e}T_{e}^{*}T_{f}T_{f}^{*}=0$. Multiplying
by $T_{f}$ on the right and by $T_{e}^{*}$ on the left, we see $T_{e}^{*}T_{e}T_{e}^{*}T_{f}T_{f}^{*}T_{f}=T_{e}^{*}T_{f}=0$,
as desired.
\end{proof}
\begin{rem}
In the proof of the previous lemma, we did not use the full power
of relation 3, only that $T_{e}T_{e}^{*}T_{f}T_{f}^{*}=0$ when $e\Lambda\cap f\Lambda=\emptyset$.
In fact, when $f=eg$, the identities above give $T_{e}T_{e}^{*}T_{f}T_{f}^{*}=T_{e}T_{g}T_{f}^{*}=T_{f}T_{f}^{*}$,
so it is enough to use the weaker form of relation 3 by only requiring
the range projections of $T_{e}$ and $T_{f}$ to be orthogonal when
$e\Lambda\cap f\Lambda=\emptyset$.

An important consequence of parts c and d in the previous result is
that every word composed of the generators $\left\{ T_{e}:e\in\Lambda\right\} $
of $\mathcal{O}\left(\Lambda\right)$ and their adjoints is of the
form $T_{e}T_{f}^{*}$.
\end{rem}
\begin{cor}
\label{cor:gens-word}If $\Lambda$ is an ordinal graph and $\mathcal{O}\left(\Lambda\right)$
has generators $\left\{ T_{e}:e\in\Lambda\right\} $, then
\[
\mathcal{O}\left(\Lambda\right)=\overline{\text{\ensuremath{\mathrm{span}}}}\left\{ T_{e}T_{f}^{*}:e,f\in\Lambda\right\} 
\]
\end{cor}
\begin{proof}
By the construction of $\mathrm{C}^{*}$-algebras which are universal
for a set of generators and relations, the span of words in the elements
$\left\{ T_{e}:e\in\Lambda\right\} \cup\left\{ T_{e}^{*}:e\in\Lambda\right\} $
is dense in $\mathcal{O}\left(\Lambda\right)$. As mentioned above,
\lemref{basic-calcs} part d implies that every such word is equal
to an element in $\left\{ T_{e}T_{f}^{*}:e,f\in\Lambda\right\} $,
and the result follows.
\end{proof}
Let $\alpha,\beta\in\mathrm{Ord}$ with $\alpha\leq\beta$ be given.
By the following lemma, each instance in $\mathcal{O}\left(\Lambda_{\alpha}\right)$
of relation 4 from \thmref{generators} continues to hold in $\mathcal{O}\left(\Lambda_{\beta}\right)$.
Likewise, relations 1-3 in $\mathcal{T}\left(\Lambda_{\alpha}\right)$
continue to hold in $\mathcal{T}\left(\Lambda_{\beta}\right)$, giving
us the inductive system in \propref{inductive-system}.
\begin{lem}
\label{lem:alpha-regular-implies-beta-regular}If $\alpha\leq\beta$
and $F\subseteq v\Lambda_{\alpha}$ is finite and exhaustive for $v$
in $\Lambda_{\alpha}$, then $F$ is exhaustive for $v$ in $\Lambda_{\beta}$.
\end{lem}
\begin{proof}
Suppose $f\in v\Lambda_{\beta}\backslash v\Lambda_{\alpha}$. Factor
$f$ as $f=gk$ where $\omega^{\alpha}>d(g)\geq\max_{e\in F}d(e)$.
Now for some $e\in F$, $e\Lambda_{\alpha}\cap g\Lambda_{\alpha}\not=\emptyset$,
so choose $h\in e\Lambda_{\alpha}\cap g\Lambda_{\alpha}$. Then $h=ea=gb$
for $a,b\in\Lambda_{\alpha}$. Let $g=e'c$ where $d(e')=d(e)$, whence
$h=ea=e'cb$. By uniqueness of factorization, $e=e'$, and $f\in e\Lambda_{\beta}=e\Lambda_{\beta}\cap f\Lambda_{\beta}$. 
\end{proof}
\begin{prop}
\label{prop:inductive-system}For each $\alpha,\beta\in\mathrm{Ord}$
with $\alpha\leq\beta$ there exist {*}-homomorphisms $\sigma_{\alpha}^{\beta}:\mathcal{T}\left(\Lambda_{\alpha}\right)\rightarrow\mathcal{T}\left(\Lambda_{\beta}\right)$
and $\rho_{\alpha}^{\beta}:\mathcal{O}\left(\Lambda_{\alpha}\right)\rightarrow\mathcal{O}\left(\Lambda_{\beta}\right)$
defined by
\begin{align*}
\sigma_{\alpha}^{\beta}\left(T_{e}\right) & =S_{e}\\
\rho_{\alpha}^{\beta}\left(T_{e}\right) & =S_{e}
\end{align*}
 for $e\in\Lambda_{\alpha}$, where $\left\{ T_{e}:e\in\Lambda\right\} $
and $\left\{ S_{e}:e\in\Lambda\right\} $ are the corresponding generators
of $\mathcal{T}\left(\Lambda\right)$ or $\mathcal{O}\left(\Lambda\right)$.
If $\alpha\leq\gamma\leq\beta$, $\sigma_{\gamma}^{\beta}\circ\sigma_{\alpha}^{\gamma}=\sigma_{\alpha}^{\beta}$
and $\rho_{\gamma}^{\beta}\circ\rho_{\alpha}^{\gamma}=\rho_{\alpha}^{\beta}$.
If $\beta$ is a limit ordinal, then
\begin{align*}
\mathcal{O}\left(\Lambda_{\beta}\right) & =\overline{\bigcup_{\alpha<\beta}\rho_{\alpha}^{\beta}\left(\mathcal{O}\left(\Lambda_{\alpha}\right)\right)}\\
\mathcal{T}\left(\Lambda_{\beta}\right) & =\overline{\bigcup_{\alpha<\beta}\sigma_{\alpha}^{\beta}\left(\mathcal{T}\left(\Lambda_{\alpha}\right)\right)}
\end{align*}
\end{prop}
\begin{proof}
The existence of the inductive systems follows from \lemref{alpha-regular-implies-beta-regular},
the universal properties of $\mathcal{T}\left(\Lambda\right)$ and
$\mathcal{O}\left(\Lambda\right)$, and the discussion above. Let
$\left\{ S_{e}:e\in\Lambda_{\beta}\right\} $ be the generators of
$\mathcal{O}\left(\Lambda_{\beta}\right)$. If $\alpha<\beta$ and
$\beta$ is a limit ordinal, $\alpha+1<\beta$. Hence if $d\left(e\right)<\beta$,
$S_{e}=\rho_{d\left(e\right)+1}^{\beta}\left(V_{e}\right)$ where
$V_{e}$ is the generator for $e$ in $\mathcal{O}\left(\Lambda_{d\left(e\right)+1}\right)$.
Thus $\cup_{\alpha<\beta}\rho_{\alpha}^{\beta}\left(\mathcal{O}\left(\Lambda_{\alpha}\right)\right)$
contains all finite sums of words in the generators, and hence is
dense in $\mathcal{O}\left(\Lambda_{\beta}\right)$. Identical logic
applies to $\mathcal{T}\left(\Lambda_{\beta}\right)$.
\end{proof}
Next we define the connected components of an ordinal graph. The connected
components of an ordinal graph are the finest partition of the morphisms
for which composable $f,g\in\Lambda$ belong to the same part. Each
connected component of an ordinal graph is also an ordinal graph.
We wish to describe the relationship between the Cuntz-Krieger algebra
of $\Lambda$ and the Cuntz-Krieger algebras of its connected components.
\begin{defn}
\label{def:connected-components}For an ordinal graph $\Lambda$,
define the relation $\preceq$ on $\Lambda$ where $e\preceq f$ if
there is $g\in\Lambda$ with $s(g)=r(f)$ and $r(g)=s(e)$. Let $\sim$
be the equivalence relation generated by $\preceq$, that is, $e\sim f$
if for every equivalence relation $\simeq$ on $\Lambda$ with the
property that $g\preceq h$ implies $g\simeq h$, $e\simeq f$. Then
$F\subseteq\Lambda$ is a connected component of $\Lambda$ if $F$
is an equivalence class for $\sim$. Let $\mathcal{F}\left(\Lambda\right)$
denote the set of connected components of $\Lambda$.
\end{defn}
\begin{prop}
\label{prop:connected-components}If $\Lambda$ is an ordinal graph,
then each $F\in\mathcal{F}\left(\Lambda\right)$ is an ordinal graph.
Moreover, we have the following isomorphism.
\[
\mathcal{O}\left(\Lambda\right)\cong\bigoplus_{F\in\mathcal{F}\left(\Lambda\right)}^{c_{0}}\mathcal{O}\left(F\right)
\]
\end{prop}
\begin{proof}
Let $F\in\mathcal{F}\left(\Lambda\right)$ and $e,f\in F$ with $s(e)=r(f)$.
Then $ef\in F$. Moreover, if $\alpha\leq d(e)$, $r\left(e_{\alpha}\right)=r\left(e\right)$
and $s\left(e^{\alpha}\right)=s\left(e\right)$, so $e_{\alpha},e^{\alpha}\in F$.
Thus $F$ is an ordinal graph.

Let $\mathcal{A}=\oplus_{F\in\mathcal{F}}\mathcal{O}\left(F\right)$.
We will construct a Cuntz-Krieger $\Lambda$-family in $\mathcal{A}$.
For $e\in\Lambda$, denote by $[e]$ the connected component of $\Lambda$
containing $e$. Let $\left\{ T_{e}:e\in F\right\} $ be the generators
of $\mathcal{O}\left(F\right)$ for each $F\in\mathcal{F}\left(\Lambda\right)$,
and let $\left\{ V_{e}:e\in\Lambda\right\} $ be the generators for
$\mathcal{O}\left(\Lambda\right)$. Then for each $e\in\Lambda$,
define $S_{e}\in\mathcal{A}$ by
\[
\left(S_{e}\right)_{F}=\begin{cases}
T_{e} & [e]=F\\
0 & \text{otherwise}
\end{cases}
\]

If $e\in\Lambda$, $[s(e)]=[e]$, so indeed $S_{e}^{*}S_{e}=S_{s(e)}$.
Likewise, if $e,f\in\Lambda$ with $s(e)=r(f)$, $[e]=[f]$ so $S_{e}S_{f}=S_{ef}$.
If $s(e)\not=r(f)$ and $[e]=[f]$, $S_{e}S_{f}=0$ since $T_{e}T_{f}=0$.
And if $[e]\not=[f]$, $S_{e}S_{f}=0$. Finally, if $G\subseteq\Lambda$
is finite and exhaustive for $v\in\Lambda_{0}$, then $G\subseteq[v]$.
Thus $\sum_{e\in G}S_{e}S_{e}^{*}=S_{v}$, and $\left\{ S_{e}:e\in\Lambda\right\} $
is a Cuntz-Krieger $\Lambda$-family. This gives a surjective {*}-homomorphism
$\mu:\mathcal{O}\left(\Lambda\right)\rightarrow\mathcal{A}$ such
that $\mu\left(V_{e}\right)=S_{e}$.

If $F\in\mathcal{F}\left(\Lambda\right)$, $\left\{ V_{e}:e\in F\right\} $
is a Cuntz-Krieger $F$-family. Thus there exists $\nu_{F}:\mathcal{O}\left(F\right)\rightarrow\mathcal{O}\left(\Lambda\right)$
where $\nu_{F}\left(T_{e}\right)=V_{e}$. Hence $\left(\mu\circ\nu_{F}\right)_{F}=\text{id}_{\mathcal{O}\left(F\right)}$,
and $\nu_{F}$ is injective. For $e\in\Lambda$, $\nu_{F}\left(\mathcal{O}\left(F\right)\right)T_{e}=0$
if $e\not\in F$, so $\left\{ \nu_{F}\left(\mathcal{O}\left(F\right)\right):F\in\mathcal{F}\left(\Lambda\right)\right\} $
is a family of pairwise orthogonal ideals which generate $\mathcal{O}\left(\Lambda\right)$.
Therefore $\mathcal{O}\left(\Lambda\right)\cong\oplus_{F\in\mathcal{F}\left(\Lambda\right)}\nu_{F}\left(\mathcal{O}\left(F\right)\right)\cong\mathcal{A}$.
\end{proof}

\section{Exhaustive Sets}

Let $\Lambda$ be a fixed ordinal graph. In this section, we wish
to understand the exhaustive sets which appear in relation 4 of \ref{thm:generators}
and are used to define $\mathcal{O}\left(\Lambda\right)$. By assuming
relation 4 only holds for some finite exhaustive sets, we recover
the relation for all other finite exhaustive sets.
\begin{defn}
For $\alpha\in\mathrm{Ord}$, $v\in\Lambda_{0}$ is an $\alpha$ \emph{source}
if there is no path $e\in v\Lambda$ such that $d\left(e\right)=\omega^{\alpha}$.
$v$ is $\alpha$ \emph{source-regular} if for every $e\in v\Lambda_{\alpha}$
there is $f\in s\left(e\right)\Lambda$ such that $d\left(f\right)=\omega^{\alpha}$.
\end{defn}
\begin{prop}
If $v\in\Lambda_{0}$ is an $\alpha$ source and $\beta\geq\alpha$,
then $v$ is a $\beta$ source.
\end{prop}
\begin{proof}
Suppose $v\in\Lambda_{0}$ and $e\in v\Lambda$ with $d\left(e\right)=\omega^{\beta}$.
Since $\omega^{\alpha}\leq\omega^{\beta}$, $r\left(e^{\omega^{\alpha}}\right)=r\left(e\right)=v$.
Thus $e^{\omega^{\alpha}}\in v\Lambda$.
\end{proof}
If $\Lambda$ is a directed graph, we only require that $\sum_{e\in v\Lambda}T_{e}T_{e}^{*}$
if $v$ is row-finite and not a source. The above definition of $\alpha$
source plays the role of sources for a directed graph. In fact, if
$\Lambda$ is a directed graph, then $v$ is a source if and only
if $v$ is a $0$ source. Likewise, the following definition is analogous
to row-finiteness for directed graphs. 
\begin{defn}
For $\alpha\in\mathrm{Ord}$, $v\in\Lambda_{0}$ is $\alpha$ \emph{row-finite}
if $\left|\left\{ f\in v\Lambda:d(f)=\omega^{\alpha}\right\} \right|<\infty$.
$v$ is $\alpha$ \emph{regular} if $v$ is $\alpha$ source-regular
and $v$ is $\alpha$ row-finite.
\end{defn}
\begin{example}
Let $\Lambda$ be the ordinal graph in \figref{one-loop-fig}. Then
$v$ is 1 source-regular and 1 row-finite, hence $v$ is 1-regular.
\end{example}
\begin{example}
\label{exa:two-loop-example}Consider an ordinal graph $\Lambda$
generated by $v,e,f,g$ where
\begin{enumerate}
\item $v=s\left(e\right)=r\left(e\right)=s\left(f\right)=r\left(f\right)=s\left(g\right)=r\left(g\right)$
\item $efg=g$
\item $d\left(v\right)=0$
\item $d\left(e\right)=d\left(f\right)=1$
\item $d\left(g\right)=\omega$
\end{enumerate}
Then $v$ is 1 source-regular because if $d\left(h\right)<\omega$
with $r\left(h\right)=v$, $s\left(h\right)=r\left(g\right)$ and
$d\left(g\right)=\omega$. However, $v$ is not $1$ row-finite because
for each $k\in\mathbb{N}$ there exists a distinct path $\underbrace{ee\ldots e}_{k\text{ times}}g$
of length $\omega$. Thus $v$ is not $1$-regular.

\begin{figure}[h]
\begin{center}
\begin{tikzpicture}[every loop/.style={looseness=60}]
	\draw[->] node[circle, fill, inner sep=0, minimum size=4pt, label=above:$v$] (v) {} edge[in=115, out=180, loop] node[left] {$e$} ();
	\draw[->] (v) edge[in=65, out=0, loop] node[right] {$f$} ();
	\draw[->] (v) edge[in=225, out=315, loop] node[below] {$g: efef\ldots;g=efg$} ();
\end{tikzpicture}
\end{center}

\caption{\label{fig:two-loop-fig}$\Lambda$ as in \exaref{two-loop-example}}
\end{figure}
\end{example}
The previous example demonstrates how $1$-regularity is much more
delicate than $0$-regularity. The source of $1$-irregularity is
that $\left|v\Lambda^{\omega}\right|=\infty$, and this happens even
when $\Lambda$ is finitely generated. On the other hand, a finitely
generated ordinal graph will never have $\left|v\Lambda^{1}\right|=\infty$,
so $0$-regularity occurs often. We dedicate most of this section
to proving that all finite exhaustive sets are obtained by considering
the paths in $v\Lambda^{\omega^{\alpha}}$ for $\alpha$ regular vertices
$v$.
\begin{prop}
\label{prop:regular-smaller}If $v$ is $\alpha$ regular and $\beta<\alpha$,
then $v$ is $\beta$ regular.
\end{prop}
\begin{proof}
Suppose $v\in\Lambda_{0}$ is $\alpha$ regular. Let $e\in v\Lambda_{\beta}$
be arbitrary. Then $e\in v\Lambda_{\alpha}$, and since $v$ is $\alpha$
regular, there exists $f\in s\left(e\right)\Lambda$ with $d\left(f\right)=\omega^{\alpha}$.
Then $f_{\omega^{\beta}}\in s\left(e\right)\Lambda$, hence $v$ is
$\beta$ source-regular. Now suppose for each $k\in\mathbb{N}$ there
are distinct $g_{k}\in v\Lambda$ with $d\left(g_{k}\right)=\omega^{\beta}$.
Since $v$ is $\alpha$ source-regular, for each $k\in\mathbb{N}$
choose $h_{k}\in s\left(g_{k}\right)\Lambda$ with $d\left(h_{k}\right)=\omega^{\alpha}$.
Then $d\left(g_{k}h_{k}\right)=d\left(g_{k}\right)+d\left(h_{k}\right)=\omega^{\beta}+\omega^{\alpha}=\omega^{\alpha}$,
and by unique factorization,
\[
\left|\left\{ g_{k}h_{k}:k\in\mathbb{N}\right\} \right|=\infty
\]
This contradicts $\alpha$ regularity of $v$.
\end{proof}
\begin{lem}
\label{lem:equal-length}If $e\Lambda\cap f\Lambda\not=\emptyset$
and $d(e)=d(f)$, then $e=f$.
\end{lem}
\begin{proof}
By \lemref{join-paths}, we may assume without loss of generality
that $e\in f\Lambda$. Then $e=fg$ for some $g\in\Lambda$. Since
$d\left(es\left(e\right)\right)=d\left(e\right)=d\left(f\right)$
and $es\left(e\right)=fg$, unique factorization implies $e=f$.
\end{proof}
Since we wish to see that $\alpha$ regular vertices play the role
of regular vertices in directed graphs, we show that in the following
sense, relation 4 is only applied to vertices which are $\alpha$-regular
for some $\alpha$. Later we will see that it suffices to only consider
exhaustive sets whose members have length which is a power of $\omega$.
\begin{thm}
\label{thm:regular-vertex-exhaustive-sets}$v\in\Lambda_{0}$ is $\alpha$
regular if and only if there exists $F\subseteq v\Lambda^{\omega^{\alpha}}$
which is finite and exhaustive for $v$. If $v$ is $\alpha$ regular,
then $v\Lambda^{\omega^{\alpha}}$ is the unique such $F\subseteq v\Lambda^{\omega^{\alpha}}$
which is finite and exhaustive.
\end{thm}
\begin{proof}
First assume $v$ is $\alpha$ regular, define $F=v\Lambda^{\omega^{\alpha}}$,
and suppose $f\in v\Lambda$. Since $v$ is $\alpha$ row-finite,
$F$ is finite. Moreover, either $d(f)<\omega^{\alpha}$ or $d(f)\geq\omega^{\alpha}$.
In the first case, $v$ is $\alpha$ source-regular, so there exists
$g\in\Lambda^{\omega^{\alpha}}$ such that $r(g)=s(f)$. Then $d(fg)=d(f)+d(g)=d(f)+\omega^{\alpha}=\omega^{\alpha}$.
Thus $fg\in F$ with $f\Lambda\cap fg\Lambda\not=\emptyset$. In the
second case, factor $f$ as $f=f_{\omega^{\alpha}}f^{\omega^{\alpha}}$
where $d\left(f_{\omega^{\alpha}}\right)=\omega^{\alpha}$. Then $f_{\omega^{\alpha}}\in F$
with $f\Lambda\cap f_{\omega^{\alpha}}\Lambda\not=0$, and $F$ is
exhaustive.

Now suppose $G\subseteq v\Lambda^{\omega^{\alpha}}$ is finite and
exhaustive for $v$. Then if $g\in\Lambda^{\omega^{\alpha}}$ with
$r(g)=v$, there is $f\in G$ such that $f\Lambda\cap g\Lambda\not=\emptyset$.
Since $d(f)=d(g)$, \lemref{equal-length} implies $f=g$. Thus $v$
is $\alpha$ row-finite, and $F=G$. If $e\in v\Lambda_{\alpha}$,
then choose $p\in G$ such that $e\Lambda\cap p\Lambda\not=\emptyset$.
Since $d(e)<d(p)$, \lemref{join-paths} implies $p\in e\Lambda$.
Let $q\in\Lambda$ with $p=eq$. Then $d(p)=\omega^{\alpha}=d(e)+d(q)$,
so $d(q)=\omega^{\alpha}$. Thus $v$ is $\alpha$ source-regular,
hence regular.
\end{proof}
\begin{defn}
If $F\subseteq v\Lambda$ is exhaustive for $v\in\Lambda_{0}$, we
say $F$ is minimal if whenever $e,f\in F$ and $e\Lambda\cap f\Lambda\not=\emptyset$,
$e=f$.

If we start with a finite exhaustive set, we may consider what happens
when we continually remove one path in every pair $e,f\in\Lambda$
with $e\Lambda\cap f\Lambda\not=\emptyset$. Eventually we get a minimal
exhaustive set, and for these sets, the join of the range projections
is the sum of the range projections, as the following two results
demonstrate:
\end{defn}
\begin{lem}
\label{lem:nonredundant-transitive}If $e\Lambda\cap f\Lambda\not=\emptyset$
and $e\in g\Lambda$, then $f\Lambda\cap g\Lambda\not=\emptyset$.
\end{lem}
\begin{proof}
Suppose $e=gh$, and first assume $f\in e\Lambda$. Then $f\in gh\Lambda\subseteq g\Lambda$,
so $f\in f\Lambda\cap g\Lambda$. Now assume $e\in f\Lambda$. Then
there is $k\in\Lambda$ such that $e=gh=fk$. Either $d(g)<d(f)$
or $d(g)\geq d(f)$. In the first case, factor $f$ as $f=pq$ where
$d(p)=d(g)$. Then by unique factorization, $p=g$ and $f\in g\Lambda$.
If $d(g)\geq d(f)$, then factor $g$ as $g=pq$ where $d(p)=d(f)$.
Unique factorization similarly implies $p=f$, so $g\in f\Lambda$.
In any case, $g\Lambda\cap f\Lambda\not=\emptyset$.
\end{proof}
\begin{lem}
\label{lem:minimal-sum-join}If $F\subseteq v\Lambda$ is finite,
minimal, and exhaustive for $v$ and $\left\{ T_{e}:e\in\Lambda\right\} $
is a Toeplitz $\Lambda$-family, then
\[
\bigvee_{e\in F}T_{e}T_{e}^{*}=\sum_{e\in F}T_{e}T_{e}^{*}
\]
\end{lem}
\begin{proof}
We will induct on subsets of $F$. Suppose $G\subsetneq F$ and
\[
\bigvee_{e\in G}T_{e}T_{e}^{*}=\sum_{e\in G}T_{e}T_{e}^{*}
\]
Let $f\in F\backslash G$. Since $F$ is minimal, we have for all
$e\in G$
\[
T_{e}T_{e}^{*}T_{f}T_{f}^{*}=0
\]
Thus
\begin{align*}
\bigvee_{e\in G\cup\{f\}}T_{e}T_{e}^{*} & =\left(\bigvee_{e\in G}T_{e}T_{e}^{*}\right)\vee T_{f}T_{f}^{*}=\left(\sum_{e\in G}T_{e}T_{e}^{*}\right)\vee T_{f}T_{f}^{*}\\
 & =\sum_{e\in G}T_{e}T_{e}^{*}+T_{f}T_{f}^{*}-\sum_{e\in G}T_{e}T_{e}^{*}T_{f}T_{f}^{*}\\
 & =\sum_{e\in G}T_{e}T_{e}^{*}+T_{f}T_{f}^{*}=\sum_{e\in G\cup\{f\}}T_{e}T_{e}^{*}
\end{align*}
\end{proof}
Now we show that having relation 4 for all finite minimal exhaustive
sets implies relation 4 holds for all finite exhaustive sets.
\begin{prop}
\label{prop:minimal-exhaustive}Suppose $\left\{ T_{e}:e\in\Lambda\right\} $
is a Toeplitz $\Lambda$-family and for every $v\in\Lambda_{0}$ and
$F\subseteq v\Lambda_{0}$ which is finite, minimal, and exhaustive
for $v$,
\[
\sum_{f\in F}T_{f}T_{f}^{*}=T_{v}
\]
Then $\left\{ T_{e}:e\in\Lambda\right\} $ is a Cuntz-Krieger $\Lambda$-family. 
\end{prop}
\begin{proof}
Suppose $F\subseteq v\Lambda$ is finite and exhaustive for $v$.
If $F$ is minimal, then indeed by \lemref{minimal-sum-join},
\[
\bigvee_{g\in F}T_{g}T_{g}^{*}=\sum_{g\in F}T_{g}T_{g}^{*}=T_{v}
\]
Otherwise, choose $e,f\in F$ such that $f\in e\Lambda$. Then
\begin{align*}
\bigvee_{g\in F}T_{g}T_{g}^{*} & =\left(\bigvee_{g\in F\backslash\{e,f\}}T_{g}T_{g}^{*}\right)\vee T_{e}T_{e}^{*}\vee T_{f}T_{f}^{*}\\
 & =\left(\bigvee_{g\in F\backslash\{e,f\}}T_{g}T_{g}^{*}\right)\vee\left(T_{e}T_{e}^{*}+T_{f}T_{f}^{*}-T_{e}T_{e}^{*}T_{f}T_{f}^{*}\right)\\
 & =\left(\bigvee_{g\in F\backslash\{e,f\}}T_{g}T_{g}^{*}\right)\vee\left(T_{e}T_{e}^{*}+T_{f}T_{f}^{*}-T_{f}T_{f}^{*}\right)=\bigvee_{g\in F\backslash\{f\}}T_{g}T_{g}^{*}
\end{align*}
By \lemref{nonredundant-transitive}, if $h\Lambda\cap f\Lambda\not=\emptyset$,
then $h\Lambda\cap e\Lambda\not=\emptyset$. Thus $F\backslash\left\{ f\right\} $
is exhaustive. If $F\backslash\{f\}$ is not minimal, then the process
above may be repeated. Since $\left|F\backslash\left\{ f\right\} \right|<\left|F\right|$
and $\left|F\right|$ is finite, the process eventually terminates
with a finite exhaustive minimal $F'$ such that
\[
\bigvee_{g\in F}T_{g}T_{g}^{*}=\bigvee_{g\in F'}T_{g}T_{g}^{*}=\sum_{g\in F'}T_{g}T_{g}^{*}=T_{v}
\]
\end{proof}
Once we have a finite minimal exhaustive set $F$ for $v$, we can
further reduce relation 4 by splitting $F$ into two smaller finite
minimal exhaustive sets $F_{e}$ and $F^{e}$. If relation 4 holds
for $F_{e}$ and $F^{e}$, then relation 4 holds for $F$, as shown
in the next two results.
\begin{prop}
Suppose $F\subseteq v\Lambda$ is minimal and exhaustive for $v$.
Let $e\in v\Lambda$ and suppose $ef\in F$ for some $f\in\Lambda$.
Define
\begin{align*}
F_{e} & =\left(F\backslash e\Lambda\right)\cup\{e\}\\
F^{e} & =\left\{ g\in\Lambda:s\left(e\right)=r\left(g\right)\text{ and }eg\in F\right\} 
\end{align*}
Then $F_{e}$ and $F^{e}$ are minimal. $F_{e}$ is exhaustive for
$v$ and $F^{e}$ is exhaustive for $s(e)$.
\end{prop}
\begin{proof}
Let $k\in v\Lambda$ be given. Choose $p\in F$ such that $p\Lambda\cap k\Lambda\not=\emptyset$.
If $p\not\in e\Lambda$, then $p\in F_{e}$. Otherwise by \lemref{nonredundant-transitive},
$e\Lambda\cap k\Lambda\not=\emptyset$. Therefore, $F_{e}$ is exhaustive.
We wish to prove $F_{e}$ is minimal. To do so, assume $p,q\in F_{e}$
and $p\Lambda\cap q\Lambda\not=\emptyset$. If $p\in F\backslash e\Lambda$
and $q=e$, then by \lemref{join-paths}, $e\in p\Lambda$. This means
$ef\in p\Lambda$, and since $F$ is minimal, $ef=p$, contradicting
$p\not\in e\Lambda$. Thus we may assume without loss of generality
that either $\left\{ p,q\right\} \subseteq F\backslash e\Lambda$
or $p=q=e$. In the first case, minimality of $F$ implies $p=q$,
so in any case we have $p=q$. Thus $F_{e}$ is minimal.

Now, let $k\in s(e)\Lambda$ be given. Choose $p\in F$ such that
$p\Lambda\cap ek\Lambda\not=\emptyset$. If $p\in ek\Lambda$, choose
$q\in\Lambda$ with $p=ekq$. Then $kq\in F^{e}$ and $kq\in k\Lambda\cap kq\Lambda$.
On the other hand, if $ek\in p\Lambda$, choose $q\in\Lambda$ with
$ek=pq$. Then $ek\in e\Lambda\cap p\Lambda$, so by \lemref{join-paths},
either $p\in e\Lambda$ or $e\in p\Lambda$. In the first case, choose
$j\in\Lambda$ such that $p=ej$. Since $ek=ejq$, $k=jq$ and $k\in j\Lambda$.
Moreover, $j\in F^{e}$ since $p\in F$. In the second case, $e\in p\Lambda$,
hence $ef\in p\Lambda$. Minimality of $F$ implies $ef=p$, hence
$ek=efq$, and $k=fq\in f\Lambda$. As $f\in F^{e}$, this implies
$F^{e}$ is exhaustive for $s(e)$.

Finally, suppose $g,h\in F^{e}$ with $g\Lambda\cap h\Lambda\not=\emptyset$.
Then $eg\Lambda\cap eh\Lambda\not=\emptyset$, so by minimality of
$F$, $eg=eh$. Then $g=h$, and $F^{e}$ is minimal.
\end{proof}
\begin{prop}
\label{prop:exhaustive-cuts-relations}Suppose $\left\{ T_{e}:e\in\Lambda\right\} $
is a Toeplitz $\Lambda$-family. Assume $v\in\Lambda_{0}$, $e\in v\Lambda$,
$F\subseteq v\Lambda$ with $F$ finite, minimal, and exhaustive for
$v$, and $F\cap e\Lambda\not=\emptyset$. If
\begin{align*}
\sum_{f\in F_{e}}T_{f}T_{f}^{*} & =T_{v}\\
\sum_{f\in F^{e}}T_{f}T_{f}^{*} & =T_{s(e)}
\end{align*}
then
\[
\sum_{f\in F}T_{f}T_{f}^{*}=T_{v}
\]
\end{prop}
\begin{proof}
We compute
\begin{align*}
\sum_{f\in F}T_{f}T_{f}^{*} & =\sum_{f\in F\cap e\Lambda}T_{f}T_{f}^{*}+\sum_{f\in F\backslash e\Lambda}T_{f}T_{f}^{*}\\
 & =\sum_{g\in F^{e}}T_{eg}T_{eg}^{*}+\sum_{f\in F_{e}}T_{f}T_{f}^{*}-T_{e}T_{e}^{*}\\
 & =T_{e}\left(\sum_{g\in F^{e}}T_{g}T_{g}^{*}\right)T_{e}^{*}-T_{e}T_{e}^{*}+\sum_{f\in F_{e}}T_{f}T_{f}^{*}\\
 & =T_{e}T_{e}^{*}-T_{e}T_{e}^{*}+\sum_{f\in F_{e}}T_{f}T_{f}^{*}=\sum_{f\in F_{e}}T_{f}T_{f}^{*}=T_{v}
\end{align*}
\end{proof}
\begin{lem}
If $\alpha\in\mathrm{Ord}$ and $F\subseteq v\Lambda^{\alpha}$ is
exhaustive for $v$, then $F$ is minimal.
\end{lem}
\begin{proof}
Suppose without loss of generality that $e,f\in F$ and $f\in e\Lambda$.
Write $f=eg$ for some $g\in\Lambda$. Then
\[
d\left(f\right)=\alpha=d\left(e\right)+d\left(g\right)=\alpha+d\left(g\right)
\]
Hence $d(g)=0$, and $eg=e$.
\end{proof}
This suggests that we consider only exhaustive sets whose members
have a fixed length. In fact, knowing relation 4 holds for these sets
implies it holds for all finite exhaustive sets.
\begin{prop}
Suppose $\left\{ T_{e}:e\in\Lambda\right\} $ is a Toeplitz $\Lambda$-family
such that for all $v\in\Lambda_{0}$, $\alpha\in\mathrm{Ord}$, and
finite $F\subseteq v\Lambda^{\alpha}$ which are exhaustive for $v$,
\[
\sum_{e\in F}T_{e}T_{e}^{*}=T_{v}
\]
Then $\left\{ T_{e}:e\in\Lambda\right\} $ is a Cuntz-Krieger $\Lambda$-family.
\end{prop}
\begin{proof}
By \propref{minimal-exhaustive}, it suffices to let $F$ be an arbitrary
finite, minimal, exhaustive set for $v$ and induct on $|F|$ to prove
that $\sum_{f\in F}T_{f}T_{f}^{*}=T_{v}$. Towards that end, note
that if $|F|=1$ and $f\in F$, then by hypothesis $T_{f}T_{f}^{*}=T_{v}$.
Now suppose that $G$ is minimal and exhaustive for $v$, $|G|=n$,
and that for all $u\in\Lambda_{0}$ with minimal and exhaustive $F$
for $u$ such that $|F|<n$, $\sum_{f\in F}T_{f}T_{f}^{*}=T_{u}$.
Let $m=\min_{g\in G}d\left(g\right)$, and label $G$ as $G=\left\{ g_{1},\ldots g_{n}\right\} $
where $d\left(g_{k}\right)\leq d\left(g_{k+1}\right)$. For each $1\leq k\leq n$,
factor $g_{k}$ as $g_{k}=p_{k}q_{k}$, where $d\left(p_{k}\right)=m$.
Define $G_{0}=G$, $G_{1}=G_{p_{1}}$, and note that since $G$ is
minimal and $g_{1}=p_{1}$, $G_{1}=G$. For $k\geq1$, define $G_{k}=\left(G_{k-1}\right)_{p_{k}}$,
$G^{k}=\left(G_{k-1}\right)^{p_{k}}$. Then $G_{n}=\left\{ p_{1},\ldots p_{n}\right\} \subseteq\Lambda^{m}$
is minimal and exhaustive for $v$, so by hypothesis
\[
\sum_{k=1}^{n}T_{p_{k}}T_{p_{k}}^{*}=T_{v}
\]
Moreover, if $1\leq k\leq n$, $G^{k}$ is minimal and exhaustive
for $s\left(p_{k}\right)$. If $\left|G^{k}\right|<n$ for $1\leq k\leq n$,
then by the induction hypothesis,
\[
\sum_{f\in G^{k}}T_{f}T_{f}^{*}=T_{s\left(p_{k}\right)}
\]
By repeatedly applying \propref{exhaustive-cuts-relations}, we see
$G_{n-1},G_{n-2},\ldots,G_{1}$ satisfy relation 4. That is, for $1\leq k\leq n$,
\[
\sum_{f\in G_{k}}T_{f}T_{f}^{*}=T_{v}
\]
Since $G_{1}=G$, we have the desired identity.

All that remains is the case when $\left|G^{k}\right|=n$ for some
$1<k\leq n$. First we prove that $g_{1}\in G_{k}$ for $1\leq k\leq n$.
Since $G_{1}=G_{p_{1}}$, $g_{1}=p_{1}$, and $g_{1}\in G$, we have
$g_{1}\in G_{1}$. Moreover if $g_{1}\in G_{k}\backslash G_{k+1}$,
then $g_{1}=p_{1}\in p_{k+1}\Lambda$. As $d\left(p_{k+1}\right)\geq d\left(p_{1}\right)$,
this implies $p_{1}=p_{k+1}$. Then $g_{k+1}=p_{k+1}q_{k+1}=g_{1}q_{k+1}$,
and because $G$ is minimal, $g_{k+1}=g_{1}$. This contradicts the
assumption that $\left|G\right|=n$, therefore we must have $g_{1}\in G_{k}$
for all $1\leq k\leq n$. Suppose for some $1<k\leq n$, $\left|G^{k}\right|=n$.
Then since $g_{1}\in G_{k-1}$, $p_{1}\in p_{k}\Lambda$. Since $d\left(p_{1}\right)\leq d\left(p_{k}\right)$,
this implies $g_{1}=p_{k}$. Hence $g_{k}\in g_{1}\Lambda$, and by
minimality of $G$, $g_{1}=g_{k}$. Once again, this contradicts $\left|G\right|=n$,
and hence $\left|G^{k}\right|<n$ for $1<k\leq n$. 
\end{proof}
Finally, we may instead only consider finite and exhaustive sets whose
members have lengths which are a fixed power of $\omega$.
\begin{thm}
\label{thm:finite-exhaustive-omega-alpha}Suppose $\left\{ T_{e}:e\in\Lambda\right\} $
is a Toeplitz $\Lambda$-family such that for all $v\in\Lambda_{0}$,
$\alpha\in\mathrm{Ord}$, and $F\subseteq v\Lambda^{\omega^{\alpha}}$
finite and exhaustive for $v$,
\[
\sum_{e\in F}T_{e}T_{e}^{*}=T_{v}
\]
Then $\left\{ T_{e}:e\in\Lambda\right\} $ is a Cuntz-Krieger $\Lambda$-family.
\end{thm}
\begin{proof}
We will prove by transfinite induction that for all $v\in\Lambda_{0}$
and $\beta\in\mathrm{Ord}$, if $G\subseteq v\Lambda^{\beta}$ is
finite and exhaustive, then
\[
\sum_{e\in G}T_{e}T_{e}^{*}=T_{v}
\]
Notice that if $\beta=0$ and $G\subseteq v\Lambda^{\beta}$ is finite
and exhaustive, then $G=\{v\}$, so the statement is true. Now suppose
the statement is true for all $\gamma<\beta$, and let $G\subseteq v\Lambda^{\beta}$
be finite and exhaustive. First, choose $\alpha\in\mathrm{Ord}$ such
that $\omega^{\alpha}\leq\beta<\omega^{\alpha+1}$, and note that
\thmref{base-expansion} implies $-\omega^{\alpha}+\beta<\beta$.
Label the elements of $G$ as $G=\left\{ g_{1},\ldots g_{n}\right\} $.
For each $1\leq k\leq n$, factor $g_{k}$ as $g_{k}=p_{k}q_{k}$
where $d\left(p_{k}\right)=\omega^{\alpha}$. Then $d\left(g_{k}\right)=\beta=d\left(p_{k}\right)+d\left(q_{k}\right)=\omega^{\alpha}+d\left(q_{k}\right)$,
hence $d\left(q_{k}\right)=-\omega^{\alpha}+\beta$. Define $G_{0}=G$,
and for $1\leq k\leq n$, $G_{k}=\left(G_{k-1}\right)_{p_{k}}$ and
$G^{k}=\left(G_{k-1}\right)^{p_{k}}$. Then $G_{n}=\left\{ p_{1},\ldots p_{n}\right\} $,
and $G^{k}\subseteq s\left(p_{k}\right)\Lambda^{-\omega^{\alpha}+\beta}$.
By hypothesis, $\sum_{e\in G_{n}}T_{e}T_{e}^{*}=T_{v}$. Moreover,
since $-\omega^{\alpha}+\beta<\beta$, the inductive hypothesis implies
that $\sum_{e\in G^{k}}T_{e}T_{e}^{*}=T_{s\left(p_{k}\right)}$ for
each $k$. Applying \propref{exhaustive-cuts-relations} repeatedly,
we see
\[
\sum_{e\in G_{k}}T_{e}T_{e}^{*}=T_{v}
\]
for each $k$, and in particular, for $k=0$.
\end{proof}
\begin{cor}
\label{cor:gen-relations}$\mathcal{O}\left(\Lambda\right)$ is the
universal $C^{*}$-algebra generated by a family of operators $\left\{ T_{e}:e\in\Lambda\right\} $
satisfying the following relations:
\end{cor}
\begin{enumerate}
\item $T_{e}^{*}T_{e}=T_{s(e)}$
\item $T_{e}T_{f}=\begin{cases}
T_{ef} & s(e)=r(f)\\
0 & \text{otherwise}
\end{cases}$
\item $T_{e}T_{e}^{*}T_{f}T_{f}^{*}=T_{f}T_{f}^{*}$ if $f\in e\Lambda$
\item $T_{v}=\sum_{e\in v\Lambda^{\omega^{\alpha}}}T_{e}T_{e}^{*}$ if $v$
is $\alpha$ regular
\end{enumerate}
\begin{proof}
By \thmref{regular-vertex-exhaustive-sets}, if $F\subseteq v\Lambda^{\omega^{\alpha}}$
is finite and exhaustive for $v$, then $F=v\Lambda^{\omega^{\alpha}}$
and $v$ is $\alpha$ regular. Thus by \thmref{finite-exhaustive-omega-alpha},
the $\mathrm{C}^{*}$-algebra generated by the relations above is
a Cuntz-Krieger $\Lambda$-family. Moreover, for every $\alpha$ regular
vertex $v$, $v\Lambda^{\omega^{\alpha}}$ is finite and exhaustive,
and thus the representation of $\mathcal{O}\left(\Lambda\right)$
is an isomorphism.
\end{proof}
Frequently, it is convenient to regard $\mathcal{O}\left(\Lambda\right)$
as a $\mathrm{C}^{*}$-algebra generated by the paths in $\Lambda^{\omega^{\alpha}}$
instead of all paths. In the following results, we give generators
and relations for this context.
\begin{lem}
\label{lem:3-20}Each $e\in\Lambda$ with $d(e)>0$ can be written
uniquely as $e=f_{1}f_{2}\ldots f_{n}$ where for each $k$, $f_{k}\in\Lambda^{\omega^{\beta_{k}}}$
and $\beta_{k}\geq\beta_{k+1}$.
\end{lem}
\begin{proof}
For existence, let $e\in\Lambda$ be given, and write $d(e)=\sum_{k=1}^{n}\omega^{\beta_{k}}$
in Cantor normal form using \thmref{base-expansion}. We will induct
on $n$. Note that if $n=1$, $d(e)=\omega^{\beta_{1}}$. Setting
$n=1$ and $f_{1}=e$, we get the desired factorization, and this
factorization is unique. Now, suppose the lemma is true for factorizations
of length $m<n$. Then
\[
d(e)=d\left(e_{\omega^{\beta_{1}}}e^{\omega^{\beta_{1}}}\right)=\omega^{\beta_{1}}+\sum_{k=2}^{n}\omega^{\beta_{k}}
\]
By the induction hypothesis, $e^{\omega^{\beta_{1}}}$ can be uniquely
factored as $f_{2}\ldots f_{n}$ where $d\left(f_{k}\right)=\omega^{\beta_{k}}$
for $n\geq k\geq2$. Setting $f_{1}=e_{\omega^{\beta_{1}}}$, we get
a factorization $e=f_{1}f_{2}\ldots f_{n}$, and by \defref{An-ordinal-graph},
this factorization is unique.
\end{proof}
\begin{prop}
\label{prop:gen-relations2}$\mathcal{O}\left(\Lambda\right)$ is
the universal $C^{*}$-algebra generated by partial isometries 
\[
\left\{ S_{v}:v\in\Lambda_{0}\right\} \cup\bigcup_{\alpha\in\mathrm{Ord}}\left\{ S_{e}:e\in\Lambda^{\omega^{\alpha}}\right\} 
\]
 satisfying the following relations:
\begin{enumerate}
\item $S_{e}^{*}S_{e}=S_{s(e)}$
\item $S_{e}S_{f}=S_{ef}$ if $s\left(e\right)=r\left(f\right)$ and $d\left(e\right)<d\left(f\right)$
\item $S_{e}^{*}S_{f}=0$ if $e\Lambda\cap f\Lambda=\emptyset$
\item $S_{v}=\sum_{e\in v\Lambda^{\omega^{\alpha}}}S_{e}S_{e}^{*}$ if $v$
is $\alpha$-regular
\end{enumerate}
\end{prop}
\begin{rem}
Relation 2 makes sense because if $\beta<\alpha$ with $e\in\Lambda^{\omega^{\beta}}$,
$f\in\Lambda^{\omega^{\alpha}}$, then $d\left(ef\right)=d\left(e\right)+d\left(f\right)=\omega^{\beta}+\omega^{\alpha}=\omega^{\alpha}$.
\end{rem}
\begin{proof}
Let $\mathcal{A}$ be the $\mathrm{C}^{*}$-algebra which is universal
for the generators and relations above. By \corref{gen-relations},
there exists a {*}-homomorphism $\eta:\mathcal{A}\rightarrow\mathcal{O}\left(\Lambda\right)$
with $\eta\left(S_{e}\right)=T_{e}$ for $e\in\Lambda_{0}\cup\bigcup_{\alpha\in\mathrm{Ord}}\Lambda^{\omega^{\alpha}}$.
We will construct an inverse $\mu:\mathcal{O}\left(\Lambda\right)\rightarrow\mathcal{A}$,
which will finish the proof. We wish to do so by identifying a family
$\mathcal{U}=\left\{ U_{e}:e\in\Lambda\right\} $ of operators which
satisfy the relations of \corref{gen-relations}. $\mu$ is an inverse
for $\eta$ if and only if $\mu\left(T_{e}\right)=S_{e}$ for $e\in\Lambda_{0}\cup\bigcup_{\alpha\in\mathrm{Ord}}\Lambda^{\omega^{\alpha}}$.
Thus for such $e$, we must define $U_{e}=S_{e}$. By \lemref{3-20}
and relation 2 in \corref{gen-relations}, this determines $\mathcal{U}$,
for if $e=f_{1}f_{2}\ldots f_{n}$ with $f_{k}\in\Lambda^{\omega^{\beta_{k}}}$,
$\beta_{k}\geq\beta_{k+1}$, then
\[
U_{e}=U_{f_{1}}U_{f_{2}}\ldots U_{f_{n}}=S_{f_{1}}S_{f_{2}}\ldots S_{f_{n}}
\]
Suppose $g=h_{1}h_{2}\ldots h_{m}$ with $h_{k}\in\Lambda^{\omega^{\alpha_{k}}}$,
$\alpha_{k}\geq\alpha_{k+1}$, and $s(g)=r(e)$. Choose the smallest
$r$ with $1\leq r\leq m$ such that $\alpha_{r}<\beta_{1}$. With
$p=h_{r}h_{r+1}\ldots h_{m}f_{1}$, $d(p)=\omega^{\beta_{1}}$ and
\begin{align*}
U_{g}U_{e} & =U_{h_{1}}\ldots U_{h_{r}}U_{h_{r+1}}\ldots U_{h_{m}}U_{f_{1}}\ldots U_{f_{n}}\\
 & =U_{h_{1}}\ldots U_{h_{r-1}}U_{p}U_{f_{2}}\ldots U_{f_{n}}
\end{align*}
Then $\alpha_{1}\geq\alpha_{2}\geq\ldots\geq\alpha_{r-1}\geq\beta_{1}\geq\ldots\geq\beta_{n}$,
so by uniqueness in \lemref{3-20}, $U_{ge}=U_{g}U_{e}$. It remains
to see that $\mathcal{U}$ satisfies the other relations of \corref{gen-relations}.
We may verify relation 1 with the following calculation:
\begin{align*}
U_{e}^{*}U_{e} & =U_{f_{n}}^{*}U_{f_{n-1}}^{*}\ldots U_{f_{1}}^{*}U_{f_{1}}\ldots U_{f_{n}}\\
 & =U_{f_{n}}^{*}\ldots U_{f_{2}}^{*}U_{s\left(f_{1}\right)}U_{f_{2}}\ldots U_{f_{n}}\\
 & =U_{f_{n}}^{*}\ldots U_{f_{2}}^{*}U_{f_{2}}\ldots U_{f_{n}}\\
 & \ldots\\
 & =U_{f_{n}}^{*}U_{f_{n}}=U_{s\left(f_{n}\right)}=U_{s(e)}
\end{align*}
For relation 3, suppose without loss of generality $d(e)\leq d(g)$
and $e\Lambda\cap g\Lambda=\emptyset$. If for $1\leq r\leq n$ and
$\beta=\sum_{k=1}^{r}\omega^{\beta_{k}}$, $e_{\beta}=g_{\beta}$,
then 
\begin{align*}
U_{e}^{*}U_{g} & =U_{e^{\beta}}^{*}U_{e_{\beta}}^{*}U_{g_{\beta}}U_{g^{\beta}}\\
 & =U_{e^{\beta}}^{*}U_{s\left(e_{\beta}\right)}U_{g^{\beta}}=U_{e^{\beta}}^{*}U_{g^{\beta}}
\end{align*}
Thus it suffices to show $U_{e^{\beta}}^{*}U_{g^{\beta}}=0$. By applying
relation 1, we may also assume without loss of generality that $e_{\omega^{\beta_{1}}}\not=g_{\omega^{\beta_{1}}}$.
Then \ref{lem:equal-length} implies $e_{\omega^{\beta_{1}}}\Lambda\cap g_{\omega^{\beta_{1}}}\Lambda=\emptyset$,
so by relation 3 of $\mathcal{A}$,
\[
U_{e}^{*}U_{g}=U_{e^{\omega^{\beta_{1}}}}^{*}U_{e_{\omega^{\beta_{1}}}}^{*}U_{g_{\omega^{\beta_{1}}}}U_{g^{\omega^{\beta_{1}}}}=0
\]
Finally, relation 4 of \corref{gen-relations} is identical to relation
4 for $\mathcal{A}$.
\end{proof}

\section{Infinite Paths}

For an ordinal graph $\Lambda$, we wish to define the infinite path
space on which we represent $\mathcal{O}\left(\Lambda\right)$. This
will prove that the generators of $\mathcal{O}\left(\Lambda\right)$
are non-zero. In particular, \thmref{Cuntz-Krieger-Uniqueness-Theorem}
is not useful if any of the vertex projections are zero, but \thmref{reg-rep}
proves this never happens.
\begin{defn}
For an ordinal graph $\Lambda$, define $\Lambda^{*}=\left\{ f\in\prod_{\beta<\alpha}\Lambda^{\beta}:\alpha\in\mathrm{Ord},f\left(\beta\right)\in f\left(\gamma\right)\Lambda\text{ if }\gamma\leq\beta\right\} $.
For $f\in\Lambda^{*}\cap\prod_{\beta<\alpha}\Lambda^{\beta}$, we
define $L\left(f\right)=\alpha$ and $r\left(f\right)=f\left(0\right)$.
\end{defn}
If $\Lambda$ is a directed graph, then $\Lambda^{*}$ is the set
of finite and infinite paths. For $e\in\Lambda$, we may define $f\in\Lambda^{*}\cap\prod_{\beta<d(e)+1}\Lambda^{\beta}$
by setting $f\left(\beta\right)=e_{\beta}$. Indeed, if $\gamma\leq\beta$,
$f\left(\beta\right)=e_{\beta}=\left(e_{\beta}\right)_{\gamma}\left(e_{\beta}\right)^{\gamma}=e_{\gamma}\left(e_{\beta}\right)^{\gamma}\in e_{\gamma}\Lambda=f\left(\gamma\right)\Lambda$.
In this way, we regard $\Lambda$ as a subset of $\Lambda^{*}$. Then
for $e\in\Lambda$, $L\left(e\right)=d(e)+1$.

$\Lambda^{*}$ is partially ordered, where $e\leq f$ if $L(e)\leq L(f)$
and for $\beta<L(e)$, $e\left(\beta\right)=f\left(\beta\right)$.
For a directed graph, the maximal elements of $\Lambda^{*}$ are the
infinite paths and the finite paths $e$ for which $s(e)$ is a source.
\begin{defn}[{cf. \cite[Definition 2.1]{KGRAPH}}]
$\Lambda^{\infty}$ is the set of maximal elements of $\Lambda^{*}$.
\end{defn}
\begin{prop}
\label{prop:extend-to-inf}For every $e\in\Lambda$ there exists $f\in\Lambda^{\infty}$
such that $L\left(f\right)>d\left(e\right)$ and $f\left(d\left(e\right)\right)=e$.
\end{prop}
\begin{proof}
Let $e\in\Lambda$ be arbitrary. Define $\Omega=\left\{ g\in\Lambda^{*}:L\left(g\right)>d\left(e\right)\text{ and }g\left(d\left(e\right)\right)=e\right\} $.
If $f$ is a maximal element of $\Omega$ and $g\in\Lambda^{*}$ with
$g\geq f$, then $g\in\Omega$. Thus every maximal element of $\Omega$
belongs to $\Lambda^{\infty}$. Hence it suffices to show $\Omega$
has a maximal element. By Zorn's lemma, we only need to show that
every totally ordered subset of $\Omega$ has an upper bound. Let
$C\subseteq\Omega$ be a totally ordered subset of $\Omega$, and
define $\alpha=\sup_{h\in C}L(h)$. We also define $f\in\Lambda^{*}\cap\prod_{\beta<\alpha}\Lambda^{\beta}$
where for $\beta<\alpha$, $h\in C$ with $L\left(h\right)>\beta$,
$f\left(\beta\right)=h\left(\beta\right)$. By construction, for every
$\beta<\alpha$ there exists $h\in C$ such that $L\left(h\right)>\beta$.
Moreover, $h\left(\beta\right)$ is independent of the choice of $h$
since $C$ is totally ordered, so $f$ is well-defined. Indeed, we
also constructed $f$ to be an upper bound for $C$.
\end{proof}
For $e\in\Lambda$ and $f\in\Lambda^{*}$ with $s(e)=r(f)$, we wish
to define a composition $ef\in\Lambda^{*}$ which is compatible with
composition in $\Lambda$. We achieve this with the following definition
and result.
\begin{defn}
If $e\in\Lambda$ and $f\in\Lambda^{*}$ with $s(e)=r(f)$, define
$ef\in\prod_{\beta<d(e)+L(f)}\Lambda^{\beta}$ by
\[
\left(ef\right)\left(\beta\right)=\begin{cases}
e_{\beta} & \beta\leq d\left(e\right)\\
ef\left(-d\left(e\right)+\beta\right) & \beta>d\left(e\right)
\end{cases}
\]
\end{defn}
By construction, $ef\in\Lambda^{*}$ if $e\in\Lambda$ and $f\in\Lambda^{*}$
with $s\left(e\right)=r\left(f\right)$. Note also that if $f,g\in\Lambda^{*}$,
$e\in\Lambda$, $s(e)=r(f)=r(g)$, and $f\leq g$, then $ef\leq eg$.
Moreover, for $f\in\Lambda^{*}$ and $\beta<L\left(f\right)$, there
is $g\in\Lambda^{*}$ with $f=f\left(\beta\right)g$, where $g\in\prod_{\gamma<-\beta+L\left(f\right)}\Lambda^{\gamma}$
is defined by
\[
g\left(\gamma\right)=f\left(\beta+\gamma\right)
\]

\begin{lem}
If $f\in\Lambda^{*}$ and $e\in\Lambda r(f)$, then $ef\in\Lambda^{\infty}$
if and only if $f\in\Lambda^{\infty}$.
\end{lem}
\begin{proof}
Suppose $g\in\Lambda^{*}$, $f\in\Lambda^{\infty}$, and $ef\leq g$.
Then $L(g)\geq L\left(ef\right)=d(e)+L\left(f\right)\geq d(e)$. Thus
$g\left(d\left(e\right)\right)=\left(ef\right)\left(d\left(e\right)\right)=e$,
and $g=eh$ where $L\left(h\right)=-d\left(e\right)+L\left(g\right)$
and $h\left(\beta\right)=g\left(d\left(e\right)+\beta\right)$. Hence
$ef\leq eh$, and $f\leq h$. Since $f\in\Lambda^{\infty}$, $h=f$
and $g=ef$.

For the other direction, suppose $ef\in\Lambda^{\infty}$ and $g\in\Lambda^{*}$
such that $f\leq g$. Then $r(f)=r(g)=s(e)$, and $ef\leq eg$. Since
$ef\in\Lambda^{\infty}$, $ef=eg$, and hence $f=g$.
\end{proof}
\begin{lem}
\label{lem:path-star-mult-assoc}If $e,f\in\Lambda$ with $s(e)=r(f)$
and $g\in\Lambda^{*}$ with $s(f)=r(g)$, $\left(ef\right)g=e\left(fg\right)$.
\end{lem}
\begin{proof}
First we verify $L\left(\left(ef\right)g\right)=d\left(ef\right)+L\left(g\right)=d\left(e\right)+d\left(f\right)+L\left(g\right)=d\left(e\right)+L\left(fg\right)=L\left(e\left(fg\right)\right)$.
Now, let $\beta<L\left(\left(ef\right)g\right)$ be given. If $\beta\leq d\left(e\right)$,
\[
\left(\left(ef\right)g\right)\left(\beta\right)=\left(ef\right)_{\beta}=e_{\beta}=\left(e\left(fg\right)\right)\left(\beta\right)
\]
If $d\left(e\right)\leq\beta\leq d\left(ef\right)=d\left(e\right)+d\left(f\right)$,
$-d\left(e\right)+\beta\leq d\left(f\right)$ and
\[
\left(\left(ef\right)g\right)\left(\beta\right)=\left(ef\right)_{\beta}=ef_{-d\left(e\right)+\beta}=e\left(fg\right)\left(-d\left(e\right)+\beta\right)=\left(e\left(fg\right)\right)\left(\beta\right)
\]
Finally, if $d\left(ef\right)\leq\beta$, then
\begin{align*}
\left(\left(ef\right)g\right)\left(\beta\right) & =\left(ef\right)g\left(-d\left(ef\right)+\beta\right)=\left(ef\right)g\left(-d(f)-d(e)+\beta\right)\\
 & =e\left(fg\right)\left(-d\left(e\right)+\beta\right)=\left(e\left(fg\right)\right)\left(\beta\right)
\end{align*}
\end{proof}
\begin{lem}
\label{lem:reg-inf-path}If $v\in\Lambda_{0}$ is $\alpha$-regular
and $f\in\Lambda^{\infty}$ with $r(f)=v$, then $L\left(f\right)>\omega^{\alpha}$.
\end{lem}
\begin{proof}
Since $v$ is $\alpha$-regular, $\left|v\Lambda^{\omega^{\alpha}}\right|<\infty$.
Thus if $L\left(f\right)\leq\omega^{\alpha}$, we may define a function
$g:\left[0,L\left(f\right)\right)\rightarrow\mathbb{N}$ where $g\left(\beta\right)=\left|f\left(\beta\right)\Lambda^{\omega^{\alpha}}\right|$.
Note that $g$ is decreasing because if $\gamma\leq\beta$, 
\[
f\left(\beta\right)\Lambda^{\omega^{\alpha}}=f\left(\gamma\right)f\left(\beta\right)^{\gamma}\Lambda^{\omega^{\alpha}}\subseteq f\left(\gamma\right)\Lambda^{\omega^{\alpha}}
\]
Thus $g$ is eventually constant. Select $\delta\in\left[0,L\left(f\right)\right)$
such that for $\beta\geq\gamma\geq\delta$, $g\left(\beta\right)=g\left(\gamma\right)$.
Since $v$ is $\alpha$ source-regular, $g\left(\delta\right)\not=0$,
so choose $e\in f\left(\delta\right)\Lambda^{\omega^{\alpha}}$. We
claim $f\leq e$. For $\beta<\delta$, we have $f\left(\beta\right)=f\left(\delta\right)_{\beta}=e_{\beta}$.
Suppose $\beta\geq\delta$ and $f\left(\beta\right)\not=e_{\beta}$.
Then $e\in f\left(\delta\right)\Lambda^{\omega^{\alpha}}\backslash f\left(\beta\right)\Lambda^{\omega^{\alpha}}$,
and $g\left(\beta\right)<g\left(\delta\right)$. This contradicts
the construction of $\delta$, so indeed $f\left(\beta\right)=e_{\beta}$.
Since $f\leq e$ and $f\in\Lambda^{\infty}$, $e=f$, and $L\left(f\right)=L\left(e\right)=d\left(e\right)+1=\omega^{\alpha}+1>\omega^{\alpha}$.
\end{proof}
\begin{thm}
\label{thm:reg-rep}If $\Lambda$ is an ordinal graph and $\mathcal{O}\left(\Lambda\right)$
has generators $\left\{ T_{e}:e\in\Lambda\right\} $, then there is
a representation $\lambda:\mathcal{O}\left(\Lambda\right)\rightarrow B\left(\ell^{2}\left(\Lambda^{\infty}\right)\right)$
such that
\[
\lambda\left(T_{e}\right)\xi_{f}=\begin{cases}
\xi_{ef} & s(e)=r(f)\\
0 & \text{otherwise}
\end{cases}
\]
In particular, $\lambda\left(T_{e}\right)\not=0$ for all $e\in\Lambda$.
\end{thm}
\begin{proof}
We will construct a family of operators $\mathcal{V}=\left\{ V_{e}:e\in\Lambda\right\} $
satisfying the relations in \corref{gen-relations}. Define $V_{e}\xi_{f}$
by the formula above, i.e. $V_{e}\xi_{f}=\xi_{ef}$ if $s(e)=r(f)$
and $V_{e}\xi_{f}=0$ otherwise. For $f,g\in\Lambda^{\infty}$,
\[
\left\langle V_{e}\xi_{f},\xi_{g}\right\rangle =\left\langle \xi_{f},V_{e}^{*}\xi_{g}\right\rangle =\begin{cases}
1 & s(e)=r(f)\text{ and }g=ef\\
0 & \text{otherwise}
\end{cases}
\]
Thus we get the following formula for $V_{e}^{*}$.
\[
V_{e}^{*}\xi_{g}=\begin{cases}
\xi_{f} & g=ef\\
0 & g\not\in e\Lambda^{\infty}
\end{cases}
\]
Now we will verify the relations in \corref{gen-relations}. For relation
1, we have $V_{e}^{*}V_{e}\xi_{f}=0$ if $r(f)\not=s(e)$ and $V_{e}^{*}V_{e}\xi_{f}=\xi_{f}$
if $r(f)=s(e)$, so $V_{e}^{*}V_{e}=V_{s\left(e\right)}$. Relation
2 follows from \lemref{path-star-mult-assoc}. For relations 3 and
4, note that
\[
V_{e}V_{e}^{*}\xi_{g}=\begin{cases}
\xi_{f} & g=eh\text{ for some }h\in\Lambda^{\infty}\\
0 & \text{otherwise}
\end{cases}
\]
If $f\in e\Lambda$ and $g=fh$ for $h\in\Lambda^{\infty}$, then
$g=e\left(f^{d(e)}h\right)$ by \lemref{path-star-mult-assoc}. Hence
$V_{e}V_{e}^{*}V_{f}V_{f}^{*}=V_{f}V_{f}^{*}$, which implies relation
3. Finally, to verify relation 4 let $v\in\Lambda_{0}$ be an $\alpha$-regular
vertex. By \lemref{reg-inf-path}, for every $f\in\Lambda^{\infty}$
with $r\left(f\right)=v$, there exists $g\in v\Lambda^{\omega^{\alpha}}$
such that $f\in g\Lambda^{\infty}$. Hence
\[
V_{v}=\sum_{g\in v\Lambda^{\omega^{\alpha}}}V_{g}V_{g}^{*}
\]
Since $\left\{ V_{e}:e\in\Lambda\right\} $ satisfies the relations
of \corref{gen-relations}, there exists a {*}-homomorphism $\lambda:\mathcal{O}\left(\Lambda\right)\rightarrow B\left(\ell^{2}\left(\Lambda^{\infty}\right)\right)$
with $\lambda\left(T_{e}\right)=V_{e}$. Then \propref{extend-to-inf}
shows that for every $e\in\Lambda$ there is $f\in\Lambda^{\infty}$
such that $r\left(f\right)=s\left(e\right)$, and hence $\lambda\left(T_{e}\right)\xi_{f}=\xi_{ef}\not=0$.
\end{proof}
\begin{cor}
\label{cor:distinct-generators}If $\Lambda$ is an ordinal graph
and $\mathcal{O}\left(\Lambda\right)$ is generated by $\left\{ T_{e}:e\in\Lambda\right\} $,
then for every $\alpha\in\mathrm{Ord}$, the map $e\mapsto T_{e}$
is injective on $\Lambda^{\alpha}$.
\end{cor}
\begin{proof}
Let $e,f\in\Lambda^{\alpha}$. If $e=fg$, then $d\left(e\right)=d\left(f\right)+d\left(g\right)$.
Subtracting $d\left(e\right)$ from both sides, we have $d\left(g\right)=0$,
hence $g=s\left(f\right)$ and $e=f$. Thus if $e\not=f$, \lemref{finite-align}
implies $e\Lambda\cap f\Lambda=\emptyset$. By \lemref{basic-calcs},
we obtain $T_{e}^{*}T_{f}=0$. Then $T_{e}\not=T_{f}$, otherwise
$T_{e}^{*}T_{f}=T_{e}^{*}T_{e}=T_{s\left(e\right)}\not=0$.
\end{proof}

\section{Correspondences}

In this section, we show that $\mathcal{O}\left(\Lambda\right)$ may
be represented using $\mathrm{C}^{*}$-correspondences in a manner
which generalizes the case for directed graphs. Instead of one correspondence,
we construct infinitely many correspondences $X_{\alpha}$ over the
algebras $\rho_{\alpha}^{\alpha+1}\left(\mathcal{O}\left(\Lambda_{\alpha}\right)\right)$
defined in \propref{inductive-system}. The correspondences are all
related by \thmref{Cuntz-Krieger-Uniqueness-Theorem}; in particular,
when $\Lambda$ is well-behaved, each $X_{\alpha+1}$ is a correspondence
over the Cuntz-Pimsner algebra $\mathcal{O}\left(X_{\alpha}\right)$. 
\begin{thm}
Let $\alpha\in\mathrm{Ord}$ be fixed and $\mathcal{A}_{\alpha}=\rho_{\alpha}^{\alpha+1}\left(\mathcal{O}\left(\Lambda_{\alpha}\right)\right)\subseteq\mathcal{O}\left(\Lambda_{\alpha+1}\right)$.
Suppose $\mathcal{O}\left(\Lambda_{\alpha+1}\right)$ is generated
by $\left\{ T_{e}:e\in\Lambda_{\alpha}\right\} $. Define
\[
X_{\alpha}=\overline{\mathrm{span}}\left\{ T_{e}T_{f}^{*}\in\mathcal{O}\left(\Lambda_{\alpha+1}\right):\omega^{\alpha}\leq d\left(e\right)<\omega^{\alpha}\cdot2,f\in\Lambda_{\alpha}\right\} 
\]
For $x,y\in X_{\alpha}$ and $a\in\mathcal{A}_{\alpha}$, define 
\begin{align*}
x\cdot a & =xa\in X_{\alpha}\\
\varphi_{\alpha}\left(a\right)x & =a\cdot x=ax\in X_{\alpha}\\
\left\langle x,y\right\rangle _{\alpha} & =x^{*}y\in\mathcal{A}_{\alpha}
\end{align*}
where $\varphi_{\alpha}:\mathcal{A}_{\alpha}\rightarrow\mathcal{L}\left(X_{\alpha}\right)$.
Then $\left(X_{\alpha},\varphi_{\alpha}\right)$ is a correspondence
over $\mathcal{A}_{\alpha}$.
\end{thm}
\begin{proof}
Fix $\alpha\in\mathrm{Ord}$. Define
\begin{align*}
\mathcal{B} & =\text{\ensuremath{\mathrm{span}}}\left\{ T_{e}T_{f}^{*}:e,f\in\Lambda_{\alpha}\right\} \\
Y & =\text{\ensuremath{\mathrm{span}}}\left\{ T_{e}T_{f}^{*}\in\mathcal{O}\left(\Lambda_{\alpha+1}\right):\omega^{\alpha}\leq d\left(e\right)<\omega^{\alpha}\cdot2,f\in\Lambda_{\alpha}\right\} 
\end{align*}
By \corref{gens-word}, $\overline{\mathcal{B}}=\mathcal{A}_{\alpha}$,
and we also have $\overline{Y}=X_{\alpha}$. First we show that $Y$
is closed under the right action of $\mathcal{B}$. If $x=T_{e}T_{f}^{*}$
for some $\omega^{\alpha}\leq d\left(e\right)<\omega^{\alpha}\cdot2$
and $f\in\Lambda_{\alpha}$, choose $g,h\in\Lambda$ such that $d\left(g\right)=\omega^{\alpha}$
and $e=gh$. Then $x=T_{g}T_{h}T_{f}^{*}$ for $h,f\in\Lambda_{\alpha}$.
If $a=T_{p}T_{q}^{*}$ for $p,q\in\Lambda_{\alpha}$, then by \lemref{basic-calcs},
$x\cdot a=T_{g}T_{h}T_{f}^{*}T_{p}T_{q}^{*}$ is either $0$ or $T_{g}T_{r}T_{s}^{*}$
for some $r,s\in\Lambda_{\alpha}$. Since $\mathcal{B}$ and $Y$
are the span of elements of these forms, it follows that $Y$ is closed
under the right action of $\mathcal{B}$.

Next we will prove the function $\left(x,y\right)\mapsto x^{*}y=\left\langle x,y\right\rangle _{\alpha}$
with domain $Y\times Y$ maps into $\mathcal{B}$. Let $x=T_{e}T_{f}^{*}$
and $y\in T_{g}T_{h}^{*}$ with $\omega^{\alpha}\leq\min\left\{ d\left(e\right),d\left(g\right)\right\} $,
$\max\left\{ d\left(e\right),d\left(g\right)\right\} <\omega^{\alpha}\cdot2$,
and $f,h\in\Lambda_{\alpha}$. Then
\[
\left\langle x,y\right\rangle _{\alpha}=\left(T_{e}T_{f}^{*}\right)^{*}T_{g}T_{h}^{*}=T_{f}T_{e}^{*}T_{g}T_{h}^{*}
\]
As $T_{f},T_{h}^{*}\in\mathcal{B}$ and $\mathcal{B}$ is a linear
subspace, it suffices to show $T_{e}^{*}T_{g}\in\mathcal{B}$. By
taking adjoints, we may assume without loss of generality $d\left(e\right)\leq d\left(g\right)$.
Then \lemref{basic-calcs} implies that either $T_{e}^{*}T_{g}=0$,
or there exists $p\in\Lambda$ such that $g=ep$ and $T_{e}^{*}T_{g}=T_{p}$.
If $g=ep$, then $\omega^{\alpha}\leq d\left(g\right)=d\left(e\right)+d\left(p\right)<\omega^{\alpha}\cdot2$.
Subtracting $\omega^{\alpha}$ from both sides, we see $-\omega^{\alpha}+d\left(e\right)+d\left(p\right)<-\omega^{\alpha}+\omega^{\alpha}\cdot2=\omega^{\alpha}$,
and in particular, $d\left(p\right)<\omega^{\alpha}$. Thus $T_{e}^{*}T_{g}=T_{p}\in\mathcal{B}$,
as desired.

The operation $\left\langle \cdot,\cdot\right\rangle _{\alpha}$ is
easily verified to be an inner product on $Y$. Moreover, the inner
product on $Y$ induces a norm which agrees with the C{*}-norm, as
the following calculation shows.
\[
\left\Vert x\right\Vert _{Y}=\left\Vert \left\langle x,x\right\rangle \right\Vert _{\mathcal{O}\left(\Lambda_{\alpha+1}\right)}^{1/2}=\left\Vert x^{*}x\right\Vert _{\mathcal{O}\left(\Lambda_{\alpha+1}\right)}^{1/2}=\left\Vert x\right\Vert _{\mathcal{O}\left(\Lambda_{\alpha+1}\right)}
\]
Additionally, the right action of $\mathcal{B}$ is continuous, hence
extends to a right action of $\mathcal{A}_{\alpha}$ on $X_{\alpha}$
which agrees with right multiplication in $\mathcal{O}\left(\Lambda_{\alpha+1}\right)$.
Similarly, the $\mathcal{B}$-valued inner product on $Y$ extends
to an inner product on $X_{\alpha}$ valued in $\mathcal{A}_{\alpha}$.

All that is left is to verify that $\varphi_{\alpha}$ is a {*}-homomorphism
into $\mathcal{L}\left(X_{\alpha}\right)$. The following calculation
does this.
\begin{align*}
\left\langle \varphi_{\alpha}\left(a\right)x,y\right\rangle _{\alpha} & =\left(\varphi_{\alpha}\left(a\right)x\right)^{*}y=\left(ax\right)^{*}y\\
 & =x^{*}a^{*}y=x^{*}\left(\varphi_{\alpha}\left(a^{*}\right)y\right)=\left\langle x,\varphi_{\alpha}\left(a^{*}\right)y\right\rangle _{\alpha}
\end{align*}
\end{proof}
\begin{example}
\label{exa:graph-correspondence}Consider $X_{\alpha}$ for $\alpha=0$.
Then $\mathcal{A}_{0}=\rho_{0}^{1}\left(\mathcal{O}\left(\Lambda_{0}\right)\right)$,
and $\mathcal{O}\left(\Lambda_{0}\right)=c_{0}\left(\Lambda_{0}\right)$.
Since for each $v\in\Lambda_{0}$ we have $T_{v}\not=0$, $\rho_{0}^{1}$
is injective. In particular, $\mathcal{A}_{0}$ is isomorphic to $c_{0}\left(\Lambda_{0}\right)$.
Moreover, if $e,f\in\Lambda$ and $d\left(f\right)=0$, $T_{e}T_{f}^{*}=T_{e}T_{f}\in\left\{ 0,T_{e}\right\} $.
Thus $X_{0}=\overline{\mathrm{span}}\left\{ T_{e}:e\in\Lambda^{1}\right\} $,
and by \corref{distinct-generators} the elements of $\left\{ T_{e}:e\in\Lambda^{1}\right\} $
are distinct. In the above proof, we have $Y=\mathrm{span}\left\{ T_{e}:e\in\Lambda^{1}\right\} $
is isomorphic as an inner product $c_{0}\left(\Lambda_{0}\right)$-module
to $c_{c}\left(\Lambda^{1}\right)$, where for $a\in c_{0}\left(\Lambda_{0}\right)$
and $x,y\in c_{c}\left(\Lambda^{1}\right)$,
\begin{align*}
\left(x\cdot a\right)\left(e\right) & =x\left(e\right)a\left(s\left(e\right)\right)\text{ for }e\in\Lambda^{1}\\
\left\langle x,y\right\rangle \left(v\right) & =\sum_{e\in\Lambda^{1}v}\overline{x\left(e\right)}y\left(e\right)\text{ for }v\in\Lambda^{0}
\end{align*}
Moreover, there is a left action $\phi_{\alpha}:c_{0}\left(\Lambda_{0}\right)\rightarrow\mathcal{L}\left(c_{c}\left(\Lambda^{1}\right)\right)$
defined by
\[
\left(\phi_{\alpha}\left(a\right)x\right)\left(e\right)=\left(a\cdot x\right)\left(e\right)=a\left(r\left(e\right)\right)x\left(e\right)\text{ for }e\in\Lambda^{1}
\]
which agrees with $\varphi_{\alpha}$ on $Y$. Thus by \cite[Proposition 3.10]{KATIDEAL},
$\mathcal{O}\left(X_{0}\right)\cong\mathcal{O}\left(\Lambda_{1}\right)$.
\end{example}
\begin{defn}
Let $\pi_{\alpha}:\mathcal{A}_{\alpha}\rightarrow\mathcal{O}\left(\Lambda_{\alpha+1}\right)$
be the inclusion, which is a {*}-homomorphism. Define $\psi_{\alpha}:X_{\alpha}\rightarrow\mathcal{O}\left(\Lambda_{\alpha}\right)$
to be the inclusion, which is an isometry.
\end{defn}
By the definition of $X_{\alpha}$, $\left(\psi_{\alpha},\pi_{\alpha}\right)$
is a Toeplitz representation of $X_{\alpha}$. Understanding the Katsura
ideal will be key to proving \thmref{Cuntz-Krieger-Uniqueness-Theorem}.
We recall the definition below.
\begin{defn}[{\cite[Definition 2.3]{KATSURA1}}]
The Katsura ideal $J_{\alpha}$ is defined as
\[
J_{\alpha}=J_{X_{\alpha}}=\varphi_{\alpha}^{-1}\left(\mathcal{K}\left(X_{\alpha}\right)\right)\cap\left(\ker\varphi_{\alpha}\right)^{\perp}
\]
\end{defn}
When $\Lambda$ is a directed graph, the Katsura ideal $J_{0}=c_{0}\left(\left\{ v\in\Lambda_{0}:0<\left|r^{-1}\left(v\right)\right|<\infty\right\} \right)$
is the ideal generated by the vertex projections of $0$ regular vertices.
We would like know in general whether $J_{\alpha}$ is generated by
the vertex projections $\left\{ T_{v}:v\text{ is }\alpha\text{ regular}\right\} $.
Unfortunately, this is more difficult to determine for $\alpha>0$
since the algebras $\mathcal{A}_{\alpha}$ are not necessarily commutative.
The following result, which will be important later, makes some progress
towards answering this question.
\begin{prop}
\label{prop:regular-compact}$v\in\Lambda_{0}$ is $\alpha$ regular
if and only if $T_{v}\in J_{\alpha}$.
\end{prop}
\begin{proof}
Let $\left\{ T_{e}:e\in\Lambda_{\alpha+1}\right\} $ be the generators
of $\mathcal{O}\left(\Lambda_{\alpha+1}\right)$. Then $\mathcal{A}_{\alpha}$
is generated by $\left\{ T_{e}:e\in\Lambda_{\alpha}\right\} $. Suppose
$T_{v}\in J_{\alpha}$ and $e\in\Lambda_{\alpha}$ with $r(e)=v$
such that $s(e)$ is an $\alpha$ source. Then for $f\in\Lambda^{\omega^{\alpha}}$,
$T_{e}T_{f}=0$. Since $Y_{\alpha}$ is the $\mathcal{B}_{\alpha}$-span
of such $T_{f}$, $T_{e}\in\ker\varphi_{\alpha}$. However, $T_{v}T_{e}=T_{e}$,
which is non-zero by \thmref{reg-rep}. This contradicts $T_{v}\in\left(\ker\varphi_{\alpha}\right)^{\perp}$,
so $v$ is $\alpha$ source-regular.

To see that $v$ is $\alpha$ row-finite, suppose $v\Lambda^{\omega^{\alpha}}$
is infinite. Then for each $f\in v\Lambda^{\omega^{\alpha}}$, $\varphi_{\alpha}\left(T_{v}\right)T_{f}=T_{f}$.
We will show that this contradicts the compactness of $\varphi_{\alpha}\left(T_{v}\right)$.
Since $\varphi_{\alpha}\left(T_{v}\right)$ is compact, choose $m\in\mathbb{N}$
and $x_{n},y_{n}\in X_{\alpha}$ such that
\[
\left\Vert \varphi_{\alpha}\left(T_{v}\right)-\sum_{n=1}^{m}\theta_{x_{n},y_{n}}\right\Vert <\frac{1}{2}
\]
For fixed $x\in X_{\alpha}$, the map $y\mapsto\theta_{x,y}$ is continuous
in the operator norm topology, as shown by the following estimate:
\[
\left\Vert \left(\theta_{x,y}-\theta_{x,y'}\right)z\right\Vert =\left\Vert x\cdot\left\langle y-y',z\right\rangle \right\Vert \leq\left\Vert x\right\Vert \left\Vert z\right\Vert \left\Vert y-y'\right\Vert 
\]
Since $Y_{\alpha}=\mathrm{span}\left\{ T_{e}T_{f}^{*}:\omega^{\alpha}\leq d(e)<\omega^{\alpha}\cdot2,f\in\Lambda_{\alpha}\right\} $
is dense in $X_{\alpha}$, for each $n$ we may choose $y_{n}'\in Y_{\alpha}$
such that$\left\Vert y_{n}-y_{n}'\right\Vert <\frac{1}{2m\left\Vert x_{n}\right\Vert }$.
Then
\[
\left\Vert \sum_{n=1}^{m}\left(\theta_{x_{n},y_{n}}-\theta_{x_{n},y_{n}'}\right)\right\Vert <\sum_{n=1}^{m}\frac{\left\Vert x_{n}\right\Vert }{2m\left\Vert x_{n}\right\Vert }=\frac{1}{2}
\]
For $g\in\Lambda^{\omega^{\alpha}}$, $\left\langle T_{e}T_{f}^{*},T_{g}\right\rangle =T_{f}T_{e}^{*}T_{g}$
is non-zero only if $e_{\omega^{\alpha}}=g$. Each $y_{n}'$ is a
finite linear combination of elements of the form $T_{e}T_{f}^{*}$,
and in particular, there are only finitely many $g\in\Lambda^{\omega^{\alpha}}$
with $\left\langle y_{n}',T_{g}\right\rangle \not=0$. Choose $g\in v\Lambda^{\omega^{\alpha}}$
such that $\theta_{x_{n},y_{n}'}T_{g}=0$ for all $1\leq n\leq m$.
Then
\begin{align*}
\left\Vert \varphi_{\alpha}\left(T_{v}\right)T_{g}-\sum_{n=1}^{m}\theta_{x_{n},y_{n}}T_{g}\right\Vert  & =\left\Vert T_{g}-\sum_{n=1}^{m}\left(\theta_{x_{n},y_{n}}-\theta_{x_{n},y_{n}'}\right)T_{g}\right\Vert \\
 & \geq\left\Vert T_{g}\right\Vert -\left\Vert \sum_{n=1}^{m}\left(\theta_{x_{n},y_{n}}-\theta_{x_{n},y_{n}'}\right)T_{g}\right\Vert >1-\frac{1}{2}=\frac{1}{2}
\end{align*}
This is a contradiction, hence $v$ is $\alpha$ row-finite.

Suppose for the converse $v$ is $\alpha$ regular. Then $T_{v}=\sum_{f\in v\Lambda^{\omega^{\alpha}}}T_{f}T_{f}^{*}$.
If $a\in\ker\varphi_{\alpha}$, this implies
\[
aT_{v}=a\sum_{f\in v\Lambda^{\omega^{\alpha}}}T_{f}T_{f}^{*}=\sum_{f\in v\Lambda^{\omega^{\alpha}}}aT_{f}T_{f}^{*}=0
\]
Therefore $T_{v}\in\left(\ker\varphi_{\alpha}\right)^{\perp}$. Moreover,
for $x\in X_{\alpha}$, 
\[
\varphi_{\alpha}\left(T_{v}\right)x=T_{v}x=\sum_{f\in v\Lambda^{\omega^{\alpha}}}T_{f}T_{f}^{*}x=\sum_{f\in v\Lambda^{\omega^{\alpha}}}T_{f}\cdot\left\langle T_{f},x\right\rangle =\sum_{f\in v\Lambda^{\omega^{\alpha}}}\theta_{T_{f},T_{f}}x
\]
From this we see $\varphi_{\alpha}\left(T_{v}\right)=\sum_{f\in v\Lambda^{\omega^{\alpha}}}\theta_{T_{f},T_{f}}$,
and $T_{v}\in\left(\ker\varphi_{\alpha}\right)^{\perp}\cap\varphi_{\alpha}^{-1}\left(\mathcal{K}\left(X_{\alpha}\right)\right)=J_{\alpha}$.
\end{proof}

\section{Cuntz-Krieger Uniqueness}

Our next goal is to identify conditions under which a {*}-homomorphism
$\pi:\mathcal{O}\left(\Lambda\right)\rightarrow\mathcal{C}$ is injective.
Our strategy is to induct on $\alpha$ and prove $\pi$ restricted
to the algebras generated by $\left\{ T_{e}:e\in\Lambda_{\alpha}\right\} $
is injective. At each step of the induction we will apply \cite[Theorem 3.9]{CKU4CPA}.
First we review the definition of a non-returning vector for a correspondence.
\begin{defn}[{\cite[Definition 3.1]{CKU4CPA}}]
If $X$ is a $\mathrm{C}^{*}$-correspondence over $A$ and $\left(\psi_{X},\pi_{A}\right):\left(X,A\right)\rightarrow\mathcal{O}\left(X\right)$
is the universal covariant representation, then $\zeta\in X^{\otimes m}$
is \emph{non-returning} if for all $1\leq n<m$ and $\xi\in X^{\otimes n}$,
\[
\psi_{X}^{\otimes m}\left(\zeta\right)^{*}\psi_{X}^{\otimes n}\left(\xi\right)\psi_{X}^{\otimes m}\left(\zeta\right)=0
\]
\end{defn}
\begin{defn}[{cf. \cite[Lemma 3.7]{GRAPHALGS}}]
\label{def:non-returning}For $\alpha\in\mathrm{Ord}$ and $n<\omega$,
$e\in\Lambda^{\omega^{\alpha}\cdot n}$ is \emph{non-returning} if
for all $f\in\Lambda$ with $\omega^{\alpha}\leq d(f)<\omega^{\alpha}\cdot n$
and $\beta<\omega^{\alpha}$, $fe^{\beta}\not\in e\Lambda$.
\end{defn}
In \lemref{non-returning-path-vector} we will show every non-returning
path of $\Lambda^{\omega^{\alpha}\cdot m}$ is associated with a non-returning
vector in $X_{\alpha}^{\otimes m}$.

\begin{defn}[{\cite[Condition (S)]{CKU4CPA}}]
A $\mathrm{C}^{*}$-correspondence $X$ over $A$ satisfies condition
(S) if for all $a\in A$ with $a\geq0$, all $n\in\mathbb{N}$, and
all $\varepsilon>0$ there is $m>n$ and non-returning $\zeta\in X^{\otimes m}$
with $\|\zeta\|=1$ such that
\[
\left\Vert \left\langle \zeta,a\zeta\right\rangle \right\Vert >\left\Vert a\right\Vert -\varepsilon
\]
\end{defn}
Returning to the context of ordinal graphs, for $\alpha\in\mathrm{Ord}$,
let $\left(\mu_{\alpha},\eta_{\alpha}\right):\left(X_{\alpha},\mathcal{A}_{\alpha}\right)\rightarrow\mathcal{O}\left(X_{\alpha}\right)$
be the universal covariant representation. For the rest of the section,
let $\left\{ T_{e}:e\in\Lambda_{\alpha+1}\right\} $ be the generators
for $\mathcal{O}\left(\Lambda_{\alpha+1}\right)$ and $\left\{ S_{e}:e\in\Lambda_{\alpha}\right\} $
be the generators for $\mathcal{O}\left(\Lambda_{\alpha}\right)$.
\begin{defn}
\label{def:chi_def}For $e\in\Lambda$ with $\omega^{\alpha}\cdot n\leq d\left(e\right)<\omega^{\alpha}\cdot(n+1)$,
define $\chi_{e}\in X_{\alpha}^{\otimes n}$ so that
\[
\chi_{e}=T_{f_{1}}\otimes T_{f_{2}}\otimes\ldots\otimes T_{f_{n}}T_{g}
\]
where $d\left(f_{k}\right)=\omega^{\alpha}$, $d\left(g\right)<\omega^{\alpha}$,
and $e=f_{1}f_{2}\ldots f_{n}g$.
\end{defn}
By \lemref{3-20}, the representation of $e$ as $f_{1}f_{2}\dots f_{n}g$
exists and is unique, hence $\chi_{e}$ is well-defined.
\begin{prop}
The family $\left\{ \mu_{\alpha}^{\otimes n}\left(\chi_{e}\right):n\in\mathbb{N},e\in\Lambda,\omega^{\alpha}\cdot n\leq d(e)<\omega^{\alpha}\cdot(n+1)\right\} $
is a Cuntz-Krieger $\Lambda_{\alpha+1}$-family.
\end{prop}
\begin{proof}
We will show that there are operators satisfiying the relations of
\propref{gen-relations2}. For $e\in\Lambda_{0}\cup\Lambda^{\omega^{\beta}}$
with $\beta<\alpha$, let $U_{e}=\eta_{\alpha}\left(T_{e}\right)$,
and if $e\in\Lambda^{\omega^{\alpha}}$, define $U_{e}=\mu_{\alpha}\left(\chi_{e}\right)$.
Relations 1 through 4 hold for the generators $\left\{ U_{e}:e\in\Lambda_{0}\cup\Lambda^{\omega^{\beta}},\beta<\alpha\right\} $
since $\eta_{\alpha}$ is a {*}-homomorphism. Thus it sufficies to
check the relations for $U_{e}$ with $e\in\Lambda^{\omega^{\alpha}}$.
Relation 1 follows because
\[
U_{e}^{*}U_{e}=\mu_{\alpha}\left(\chi_{e}\right)^{*}\mu_{\alpha}\left(\chi_{e}\right)=\eta_{\alpha}\left(\left\langle \chi_{e},\chi_{e}\right\rangle \right)=\eta_{\alpha}\left(T_{e}^{*}T_{e}\right)=\eta_{\alpha}\left(T_{s(e)}\right)=U_{s\left(e\right)}
\]
For relation 2, suppose $f\in\Lambda_{0}\cup\Lambda^{\omega^{\beta}}$
for $\beta<\alpha$. Then if $s(f)=r(e)$,
\[
U_{f}U_{e}=\eta_{\alpha}\left(T_{f}\right)\mu_{\alpha}\left(\chi_{e}\right)=\mu_{\alpha}\left(\varphi_{\alpha}\left(T_{f}\right)\chi_{e}\right)=\mu_{\alpha}\left(\chi_{fe}\right)=U_{fe}
\]
There are 2 cases for relation 3. First we suppose that $f\in\Lambda_{0}\cup\Lambda^{\omega^{\beta}}$,
$\beta<\alpha$, and $e\Lambda\cap f\Lambda=\emptyset$, in which
case
\[
U_{f}^{*}U_{e}=\eta_{\alpha}\left(T_{f}^{*}\right)\mu_{\alpha}\left(\chi_{e}\right)=\mu_{\alpha}\left(\varphi_{\alpha}\left(T_{e}^{*}\right)\chi_{e}\right)=\mu_{\alpha}\left(0\right)=0
\]
For the other case, suppose $f\in\Lambda^{\omega^{\alpha}}$ and $e\Lambda\cap f\Lambda=\emptyset$.
Then
\[
U_{f}^{*}U_{e}=\mu_{\alpha}\left(\chi_{f}\right)^{*}\mu_{\alpha}\left(\chi_{e}\right)=\eta_{\alpha}\left(\left\langle \chi_{f},\chi_{e}\right\rangle \right)=\eta_{\alpha}\left(T_{f}^{*}T_{e}\right)=\eta_{\alpha}\left(0\right)=0
\]
Finally, relation 4 follows from covariance of $\left(\mu_{\alpha},\eta_{\alpha}\right)$.
In particular by \propref{regular-compact}, if $v$ is $\alpha$
regular, $T_{v}\in J_{\alpha}$. Then by covariance,
\begin{align*}
\eta_{\alpha}\left(T_{v}\right) & =U_{v}=\left(\mu_{\alpha},\eta_{\alpha}\right)^{(1)}\left(\varphi_{\alpha}\left(T_{v}\right)\right)=\left(\mu_{\alpha},\eta_{\alpha}\right)^{(1)}\left(\sum_{f\in v\Lambda^{\omega^{\alpha}}}\theta_{\chi_{f},\chi_{f}}\right)\\
 & =\sum_{f\in v\Lambda^{\omega^{\alpha}}}\left(\mu_{\alpha},\eta_{\alpha}\right)^{(1)}\left(\theta_{\chi_{f},\chi_{f}}\right)=\sum_{f\in v\Lambda^{\omega^{\alpha}}}\mu_{\alpha}\left(\chi_{f}\right)\mu_{\alpha}\left(\chi_{f}\right)^{*}=\sum_{f\in v\Lambda^{\omega^{\alpha}}}U_{f}U_{f}^{*}
\end{align*}
\end{proof}
\begin{defn}
Let $j_{\alpha}:\mathcal{O}\left(\Lambda_{\alpha+1}\right)\rightarrow\mathcal{O}\left(X_{\alpha}\right)$
be the surjective {*}-homomorphism induced by the universal property
of $\mathcal{O}\left(\Lambda_{\alpha+1}\right)$.

Note then that for $e\in\Lambda^{\omega^{\alpha}}$, $j_{\alpha}\left(T_{e}\right)=\mu_{\alpha}\left(\chi_{e}\right)$,
and for $e\in\Lambda_{\alpha}$, $j_{\alpha}\left(T_{e}\right)=\eta_{\alpha}\left(T_{e}\right)$.
\end{defn}
\begin{cor}
\label{cor:rep-calcs}If $n,m\in\mathbb{N}$ with $\omega^{\alpha}\cdot n\leq d(e)<\omega^{\alpha}\cdot(n+1)$
and $\omega^{\alpha}\cdot m\leq d(f)<\omega^{\alpha}\cdot(m+1)$,
then 
\[
\mu_{Y_{\alpha}}^{\otimes n}\left(\chi_{e}\right)^{*}\mu_{Y_{\alpha}}^{\otimes m}\left(\chi_{f}\right)=\begin{cases}
\mu_{\alpha}^{\otimes(m-n)}\left(\chi_{g}\right) & f=eg\\
\mu_{\alpha}^{\otimes(n-m)}\left(\chi_{g}\right)^{*} & e=fg\\
0 & e\Lambda\cap f\Lambda=\emptyset
\end{cases}
\]
\end{cor}
\begin{proof}
Fix $n,m\in\mathbb{N}$ with $\omega^{\alpha}\cdot n\leq d(e)<\omega^{\alpha}\cdot(n+1)$
and $\omega^{\alpha}\cdot m\leq d(f)<\omega^{\alpha}\cdot(m+1)$.
By \lemref{basic-calcs},
\begin{align*}
j_{\alpha}\left(T_{e}^{*}T_{f}\right) & =j_{\alpha}\left(T_{e}\right)^{*}j_{\alpha}\left(T_{f}\right)=\mu_{\alpha}^{\otimes n}\left(\chi_{e}\right)^{*}\mu_{\alpha}^{\otimes m}\left(\chi_{f}\right)\\
 & =\begin{cases}
j_{\alpha}\left(T_{g}\right) & f=eg\\
j_{\alpha}\left(T_{g}\right)^{*} & e=fg\\
0 & e\Lambda\cap f\Lambda=\emptyset
\end{cases}\\
 & =\begin{cases}
\mu_{\alpha}^{\otimes(m-n)}\left(\chi_{g}\right) & f=eg\\
\mu_{\alpha}^{\otimes(n-m)}\left(\chi_{g}\right)^{*} & e=fg\\
0 & e\Lambda\cap f\Lambda=\emptyset
\end{cases}
\end{align*}
\end{proof}
\begin{lem}
\label{lem:non-return-diff-mult}Let $(X,\varphi)$ be a $C^{*}$-correspondence
over $A$ with universal covariant representation $\left(\psi_{X},\pi_{A}\right):\left(X,A\right)\rightarrow\mathcal{O}\left(X\right)$.
If $\zeta\in X^{\otimes m}$ is non-returning, then for all $1\leq n<m$,
$\xi\in X^{\otimes n}$, and $a,a',b,b'\in A$,
\[
\psi_{X}^{\otimes m}\left(\varphi_{m}\left(a\right)\zeta b\right)^{*}\psi_{X}^{\otimes n}\left(\xi\right)\psi_{X}^{\otimes m}\left(\varphi_{m}\left(a'\right)\zeta b'\right)=0
\]
\end{lem}
\begin{proof}
Let $1\leq n<m$ and $\xi\in X^{\otimes n}$ be given. Then since
$\varphi_{n}\left(a^{*}\right)\xi a'\in X^{\otimes n}$, 
\begin{align*}
 & \psi_{X}^{\otimes m}\left(\varphi_{m}(a)\zeta b\right)^{*}\psi_{X}^{\otimes n}\left(\xi\right)\psi_{X}^{\otimes m}\left(\varphi_{m}(a')\zeta b'\right)\\
= & \pi_{A}\left(b\right)^{*}\psi_{X}^{\otimes m}\left(\zeta\right)^{*}\pi_{A}\left(a\right)^{*}\psi_{X}^{\otimes n}\left(\xi\right)\pi_{A}\left(a'\right)\psi_{X}^{\otimes m}\left(\zeta\right)\pi_{A}\left(b'\right)\\
= & \pi_{A}\left(b\right)^{*}\psi_{X}^{\otimes m}\left(\zeta\right)^{*}\psi_{X}^{\otimes n}\left(\varphi_{n}\left(a^{*}\right)\xi a'\right)\psi_{X}^{\otimes m}\left(\zeta\right)\pi_{A}\left(b'\right)\\
= & \pi_{A}\left(b\right)^{*}0\pi_{A}\left(b'\right)=0
\end{align*}
\end{proof}
\begin{lem}
\label{lem:non-returning-path-vector}If $e\in\Lambda^{\omega^{\alpha}\cdot m}$
is non-returning, then $\chi_{e}\in X_{\alpha}^{\otimes m}$ is non-returning.
\end{lem}
\begin{proof}
Let $1\leq n<m$ be given. Since $X_{\alpha}^{\otimes n}$ is generated
by 
\[
\left\{ \chi_{f}\cdot T_{g}^{*}:g\in\Lambda_{\alpha},f\in\Lambda,\omega^{\alpha}\cdot n\leq d(f)<\omega^{\alpha}\cdot(n+1)\right\} 
\]
it suffices to check $\mu_{\alpha}^{\otimes m}\left(\chi_{e}\right)^{*}\mu_{\alpha}^{\otimes n}\left(\chi_{f}\cdot T_{g}^{*}\right)\mu_{\alpha}^{\otimes m}\left(\chi_{e}\right)=0$.
If 
\[
\mu_{\alpha}^{\otimes n}\left(\chi_{f}\cdot T_{g}^{*}\right)\mu_{\alpha}^{\otimes m}\left(\chi_{e}\right)=\mu_{\alpha}^{\otimes n}\left(\chi_{f}\right)\mu_{\alpha}^{\otimes m}\left(\varphi_{\alpha}\left(T_{g}^{*}\right)\chi_{e}\right)\not=0
\]
then $e\in g\Lambda$, $s(f)=r\left(e^{d(g)}\right)$, and
\[
\mu_{\alpha}^{\otimes n}\left(\chi_{f}\right)\mu_{\alpha}^{\otimes m}\left(\chi_{e^{d(g)}}\right)=\mu_{\alpha}^{\otimes\left(n+m\right)}\left(\chi_{fe^{d(g)}}\right)
\]
Since $e$ is non-returning, $fe^{d(g)}\not\in e\Lambda$. Thus by
\corref{rep-calcs},
\[
\mu_{\alpha}^{\otimes m}\left(\chi_{e}\right)^{*}\mu_{\alpha}^{\otimes n}\left(\chi_{f}\cdot T_{g}^{*}\right)\mu_{\alpha}^{\otimes m}\left(\chi_{e}\right)=\mu_{\alpha}^{\otimes m}\left(\chi_{e}\right)^{*}\mu_{\alpha}^{\otimes(n+m)}\left(\chi_{fe^{d(g)}}\right)=0
\]
\end{proof}
\begin{lem}
\label{lem:nonreturning-submodule}Let $\left(X,\varphi\right)$ be
a correspondence over a $C^{*}$-algebra $A$. If $\zeta\in X^{\otimes m}$
is non-returning, then
\[
K=\overline{\text{\ensuremath{\mathrm{span}}}}\left\{ \varphi\left(a\right)\zeta\cdot b:a,b\in A\right\} 
\]
is a closed Hilbert submodule of $X^{\otimes m}$, and each $\eta\in K$
is non-returning.
\end{lem}
\begin{proof}
Let $\left(\psi,\pi\right):\left(X,\varphi\right)\rightarrow\mathcal{O}\left(X\right)$
be the universal covariant representation of $X$. Define $Y=\text{\ensuremath{\mathrm{span}}}\left\{ \varphi\left(a\right)\zeta\cdot b:a,b\in A\right\} $.
Then $Y$ is closed under the right action of $A$, hence $Y$ is
a submodule of $X$. Since the right action is continuous, $\overline{Y}=K$
is a closed Hilbert submodule of $X$. For $1\leq n<m$ and $\xi\in X^{\otimes n}$,
the function 
\[
\nu\mapsto\psi^{\otimes m}\left(\nu\right)^{*}\psi^{\otimes n}\left(\xi\right)\psi^{\otimes m}\left(\nu\right)
\]
is continuous. Thus the set of non-returning vectors in $X^{\otimes m}$
is closed, and it suffices to show that every vector in $Y$ is non-returning.
Let $\kappa\in Y$ be arbitrary, and choose $a_{k},b_{k}\in A$ with
\[
\kappa=\sum_{k=1}^{r}\varphi\left(a_{k}\right)\zeta\cdot b_{k}
\]
Since we wish to show $\kappa$ is non-returning, let $1\leq n<m$
and $\xi\in X^{\otimes n}$ be given. Then
\[
\psi^{\otimes m}\left(\kappa\right)^{*}\psi^{\otimes n}\left(\xi\right)\psi^{\otimes m}\left(\kappa\right)=\sum_{j,k=1}^{r}\psi^{\otimes m}\left(\varphi\left(a_{j}\right)\zeta\cdot b_{j}\right)^{*}\psi^{\otimes n}\left(\xi\right)\psi^{\otimes m}\left(\varphi\left(a_{k}\right)\zeta\cdot b_{k}\right)
\]
By \lemref{non-return-diff-mult}, this is $0$, and $\kappa$ is
non-returning.
\end{proof}
We are now in the position to prove a version of the Cuntz-Krieger
uniqueness theorem for ordinal graphs. To do so, we introduce what
we call condition (S) for ordinal graphs, which represents the circumstance
in which each correspondence $X_{\alpha}$ satisfies condition (S).
\begin{defn}
\label{def:alpha-full}A path $e\in\Lambda$ is \emph{$\alpha$-full}
if $e\in\Lambda\backslash\Lambda_{\alpha}$ and for every $v\in\Lambda_{0}$
in the same connected component of $\Lambda_{\alpha}$ as $r\left(e\right)$,
there exists $\beta<\omega^{\alpha}$ and $f\in\Lambda_{\alpha}$
such that $s\left(f\right)=r\left(e^{\beta}\right)$ and $r\left(f\right)=v$.
\end{defn}
\begin{rem}
Since the connected componenets of $\Lambda_{0}$ are the vertices,
every path $e\in\Lambda\backslash\Lambda_{0}$ is $0$-full. In particular,
all edges in a directed graph are $0$-full.
\end{rem}
\begin{defn}
\label{def:condition-S}An ordinal graph $\Lambda$ satisfies condition
(S) if for every $\alpha\in\mathrm{Ord}$ such that $\Lambda^{\omega^{\alpha}}\not=\emptyset$,
every connected component $F$ of $\Lambda_{\alpha}$, and every $n\in\mathbb{N}$
there exists non-returning, $\alpha$-full $f\in\Lambda$ with $r\left(f\right)\in F$
and $\omega^{\alpha}\cdot n\leq d\left(f\right)<\omega^{\alpha+1}$.
\end{defn}
\begin{rem}
If $\Lambda$ satisfies condition (S) and $\Lambda^{\omega^{\alpha}}\not=\emptyset$,
then for each $v\in\Lambda_{0}$ there exists $f\in\Lambda_{\alpha}$,
$\alpha$-full $e\in\Lambda$, and $\beta<\omega^{\alpha}$ such that
$s\left(f\right)=r\left(e^{\beta}\right)$ and $r\left(f\right)=v$.
Then $r\left(fe^{\beta}\right)=v$ and $d\left(fe^{\beta}\right)=d\left(f\right)-\beta+d\left(e\right)\geq\omega^{\alpha}$.
Thus $\Lambda$ has no $\alpha$-sources.
\end{rem}
\begin{thm}[Cuntz-Krieger Uniqueness Theorem for Ordinal Graphs]
\label{thm:Cuntz-Krieger-Uniqueness-Theorem}If $\Lambda$ is an
ordinal graph with no 1-regular vertices which satisfies condition
(S) and $\Lambda^{\omega^{\alpha}}\not=\emptyset$, then
\begin{enumerate}
\item $J_{\alpha}=0$ if $\alpha\geq1$.
\item $\mathcal{O}\left(X_{\alpha}\right)\cong\mathcal{O}\left(\Lambda_{\alpha+1}\right)$.
\item $X_{\alpha}$ satisfies condition (S).
\item If $\pi:\mathcal{O}\left(\Lambda_{\alpha+1}\right)\rightarrow\mathcal{C}$
is a {*}-homomorphism into a $C^{*}$-algebra $\mathcal{C}$ such
that for every $v\in\Lambda_{0}$, $\pi\left(T_{v}\right)\not=0$,
then $\pi$ is injective.
\end{enumerate}
In particular, if $\Lambda$ satisfies condition (S) and has no 1-regular
vertices, then a {*}-homomorphism $\pi:\mathcal{O}\left(\Lambda\right)\rightarrow\mathcal{C}$
is injective iff $\pi\left(T_{v}\right)\not=0$ for all $v\in\Lambda_{0}$.
\end{thm}
\begin{proof}
The proof is by transfinite induction. Let $\alpha$ with $\Lambda^{\omega^{\alpha}}$
be fixed, $\left\{ S_{e}:e\in\Lambda_{\alpha}\right\} $ be the generators
of $\mathcal{O}\left(\Lambda_{\alpha}\right)$, and $\left\{ T_{e}:e\in\Lambda_{\alpha+1}\right\} $
be the generators of $\mathcal{O}\left(\Lambda_{\alpha+1}\right)$.
Suppose we have Cuntz-Krieger uniqueness for $\alpha$, or more specifically,
that for every ordinal graph $\Gamma$ with $\mathcal{O}\left(\Gamma_{\alpha}\right)$
generated by $\left\{ W_{e}:e\in\Gamma_{\alpha}\right\} $ and every
{*}-homomorphism $\pi:\mathcal{O}\left(\Gamma_{\alpha}\right)\rightarrow\mathcal{C}$
with $\pi\left(W_{v}\right)\not=0$ for all $v\in\Gamma_{0}$, $\pi$
is injective. In particular, we have Cuntz-Krieger uniqueness for
$\Lambda_{\alpha}$. Note that when $\alpha=0$, $\mathcal{O}\left(\Lambda_{\alpha}\right)\cong c_{0}\left(\Lambda_{0}\right)$
has Cuntz-Krieger uniqueness, and this is the base case. First we
will prove statements 1, 2, and 3 in the theorem, and then prove that
$\mathcal{O}\left(\Lambda_{\alpha+1}\right)$ has the Cuntz-Krieger
uniqueness property. Afterwards, we will prove that $\mathcal{O}\left(\Lambda_{\beta}\right)$
has Cuntz-Krieger uniqueness when $\beta$ is a limit ordinal, which
will complete the proof. 

To see 1, suppose $\alpha\geq1$. Note that $\Lambda$ has no $1$-regular
vertices, so by \propref{regular-compact} and \propref{regular-smaller},
$T_{v}\not\in J_{\alpha}$ for every $v\in\Lambda_{0}$. Let $\mathcal{I}=\left(\rho_{\alpha}^{\alpha+1}\right)^{-1}\left(J_{\alpha}\right)\vartriangleleft\mathcal{O}\left(\Lambda_{\alpha}\right)$,
and $q:\mathcal{O}\left(\Lambda_{\alpha}\right)\rightarrow\mathcal{O}\left(\Lambda_{\alpha}\right)/\mathcal{I}$
be the quotient map. Since $\rho_{\alpha}^{\alpha+1}\left(S_{v}\right)=T_{v}$,
$S_{v}\not\in\mathcal{I}$ for all $v\in\Lambda_{0}$. Then $q\left(S_{v}\right)\not=0$
for all $v\in\Lambda_{0}$, and by hypothesis, $q$ is injective.
Therefore $\mathcal{I}=0$, and since $J_{\alpha}\subseteq\rho_{\alpha}^{\alpha+1}\left(\mathcal{O}\left(\Lambda_{\alpha}\right)\right)$,
$J_{\alpha}=\rho_{\alpha}^{\alpha+1}\left(\mathcal{I}\right)=0$. 

For statement 2, we handle the cases when $\alpha=0$ and $\alpha\geq1$
separately. If $\alpha=0$, then $\Lambda_{\alpha+1}=\Lambda_{1}$
is a directed graph, and as in \exaref{graph-correspondence}, \cite[Proposition 3.10]{KATIDEAL}
implies $\mathcal{O}\left(X_{0}\right)\cong\mathcal{O}\left(\Lambda_{1}\right)$.
If $\alpha\geq1$, then $J_{\alpha}=0$. Since $\left(\psi_{\alpha},\pi_{\alpha}\right)$
is a Toeplitz representation of $X_{\alpha}$, this implies $\left(\psi_{\alpha},\pi_{\alpha}\right)$
is Cuntz-Pimsner covariant. Covariance then induces a {*}-homomorphism
$\psi_{\alpha}\times\pi_{\alpha}:\mathcal{O}\left(X_{\alpha}\right)\rightarrow\mathcal{O}\left(\Lambda_{\alpha+1}\right)$.
We claim $\psi_{\alpha}\times\pi_{\alpha}$ is an inverse for $j_{\alpha}$.
Proving this is simple since we may check the composition on generators.
For $e\in\Lambda^{\omega^{\alpha}}$,
\[
\left(\psi_{\alpha}\times\pi_{\alpha}\circ j_{\alpha}\right)\left(T_{e}\right)=\left(\psi_{\alpha}\times\pi_{\alpha}\right)\left(\mu_{\alpha}\left(\chi_{e}\right)\right)=\psi_{\alpha}\left(\chi_{e}\right)=T_{e}
\]
Likewise for $e\in\Lambda_{\alpha}$,
\[
\left(\psi_{\alpha}\times\pi_{\alpha}\circ j_{\alpha}\right)\left(T_{e}\right)=\left(\psi_{\alpha}\times\pi_{\alpha}\right)\left(\eta_{\alpha}\left(T_{e}\right)\right)=\pi_{\alpha}\left(T_{e}\right)=T_{e}
\]
Thus $\psi_{\alpha}\times\pi_{\alpha}\circ j_{\alpha}=\text{id}$,
and $j_{\alpha}$ is injective. Since $\mathcal{O}\left(X_{\alpha}\right)$
is generated by elements of the form $\mu_{\alpha}\left(\chi_{e}\right)$
and $\eta_{\alpha}\left(T_{e}\right)$, $j_{\alpha}$ is also surjective.
Hence $j_{\alpha}$ is an isomorphism.

For statement 3, let $\mathcal{F}$ be the set of all connected components
of $\Lambda_{\alpha}$. Note that $\rho_{\alpha}^{\alpha+1}\left(S_{v}\right)=T_{v}\not=0$
for $v\in\Lambda_{0}$, so $\rho_{\alpha}^{\alpha+1}$ is injective.
Then by \propref{connected-components},
\[
\mathcal{A}_{\alpha}=\rho_{\alpha}^{\alpha+1}\left(\mathcal{O}\left(\Lambda_{\alpha}\right)\right)\cong\bigoplus_{F\in\mathcal{F}}\rho_{\alpha}^{\alpha+1}\left(\mathcal{O}\left(F\right)\right)
\]
Note that the operation above is the $c_{0}$ direct sum. For $a\in\mathcal{A}_{\alpha}$
with $a\geq0$, $n\in\mathbb{N}$, and $\varepsilon>0$ we wish to
show there exists $m\geq n$ and non-returning $\xi\in X_{\alpha}^{\otimes m}$
such that $\|\xi\|=1$ and
\[
\left\Vert \left\langle \xi,\varphi_{\alpha}\left(a\right)\xi\right\rangle _{\alpha}\right\Vert >\|a\|-\varepsilon
\]
We begin by letting $a\in\mathcal{A}_{\alpha}$, $a\geq0$, $n\in\mathbb{N}$,
and $\varepsilon>0$ be given. By the isomorphism above, choose $F\in\mathcal{F}$
such that $\left\Vert a\right\Vert =\left\Vert p_{F}\left(a\right)\right\Vert $,
where $p_{F}:\mathcal{A}_{\alpha}\rightarrow\rho_{\alpha}^{\alpha+1}\left(\mathcal{O}\left(F\right)\right)$
projects onto the $F$-component of the direct sum. Since $\Lambda_{\alpha+1}$
satisfies condition (S), there exists $m\geq n$ and non-returning,
$\alpha$-full $f\in\Lambda_{\alpha+1}$ such that $r\left(f\right)\in F$
and $\omega^{\alpha}\cdot m\leq d\left(f\right)<\omega^{\alpha}\cdot\left(m+1\right)$.
Then by \lemref{non-returning-path-vector}, $\chi_{f}\in X_{\alpha}^{\otimes m}$
is non-returning. Applying \lemref{nonreturning-submodule}, we obtain
the following closed Hilbert submodule $K$ of $X_{\alpha}^{\otimes m}$
consisting only of non-returning vectors.
\[
K=\overline{\text{\ensuremath{\mathrm{span}}}}\left\{ \varphi_{\alpha}\left(a\right)\chi_{f}\cdot b:a,b\in\mathcal{A}_{\alpha}\right\} 
\]
Since $K$ is closed under the left action $\varphi_{\alpha}$, $\varphi_{\alpha}$
restricts to a {*}-homomorphism $\tau:\rho_{\alpha}^{\alpha+1}\left(\mathcal{O}\left(F\right)\right)\rightarrow\mathcal{L}\left(K\right)$.
Here we apply the $\alpha$-fullness of $f$ to show that for each
$w\in\Lambda_{0}\cap F$, $\left(\tau\circ\rho_{\alpha}^{\alpha+1}\right)\left(S_{w}\right)=\tau\left(T_{w}\right)\not=0$.
For $w\in\Lambda_{0}\cap F$, select $\gamma<\omega^{\alpha}$ and
$e\in F$ such that $r\left(f^{\gamma}\right)=s\left(e\right)$ and
$r\left(e\right)=w$. Then $\varphi_{\alpha}\left(T_{e}T_{f_{\gamma}}^{*}\right)\chi_{f}\in K$
is non-returning, and
\[
\tau\left(T_{w}\right)\left(\varphi_{\alpha}\left(T_{e}T_{f_{\gamma}}^{*}\right)\chi_{f}\right)=\tau\left(T_{w}\right)\left(\chi_{ef^{\gamma}}\right)=\chi_{ef^{\gamma}}
\]
Moreover, $\chi_{ef^{\gamma}}\not=0$ because 
\[
\left(\mu_{\alpha}^{\otimes m}\circ j_{\alpha}^{-1}\right)\left(\chi_{ef^{\gamma}}\right)=T_{ef^{\gamma}}\not=0
\]
In particular, $\tau\left(T_{w}\right)\not=0$. 

Since $F$ satisfies condition (S), the induction hypothesis implies
$\tau$ is injective, and in particular isometric. Choose $b\in\mathcal{A}_{\alpha}$
such that $a=b^{*}b$. Then 
\[
\left\Vert b\right\Vert =\left\Vert a\right\Vert ^{1/2}=\left\Vert p_{F}\left(a\right)\right\Vert ^{1/2}=\left\Vert p_{F}\left(b\right)\right\Vert =\left\Vert \left(\tau\circ p_{F}\right)\left(b\right)\right\Vert 
\]
Using continuity of $t\mapsto t^{2}$, we choose $\delta>0$ such
that $\left(\left\Vert b\right\Vert -\delta\right)^{2}>\left\Vert b\right\Vert ^{2}-\varepsilon=\left\Vert a\right\Vert -\varepsilon$.
Then we choose $\xi\in K$ such that $\left\Vert \xi\right\Vert =1$
and $\left\Vert \left(\tau\circ p_{F}\right)\left(b\right)\xi\right\Vert >\left\Vert b\right\Vert -\delta$.
This implies
\begin{align*}
\left\Vert \left\langle \xi,\varphi_{\alpha}\left(a\right)\xi\right\rangle _{\alpha}\right\Vert  & =\left\Vert \left\langle \varphi_{\alpha}\left(b\right)\xi,\varphi_{\alpha}\left(b\right)\xi\right\rangle _{\alpha}\right\Vert \\
 & =\left\Vert \varphi_{\alpha}\left(b\right)\xi\right\Vert ^{2}\\
 & =\left\Vert \left(\tau\circ p_{F}\right)\left(b\right)\xi\right\Vert ^{2}\\
 & >\left(\left\Vert b\right\Vert -\delta\right)^{2}\\
 & >\left\Vert a\right\Vert -\varepsilon
\end{align*}
Thus $X_{\alpha}$ satisfies condition (S).

Now we prove that $\mathcal{O}\left(\Lambda_{\alpha+1}\right)\cong\mathcal{O}\left(X_{\alpha}\right)$
has Cuntz-Krieger uniqueness. Let $\pi:\mathcal{O}\left(\Lambda_{\alpha+1}\right)\rightarrow\mathcal{C}$
be a {*}-homomorphism into a $\mathrm{C}^{*}$-algebra $\mathcal{C}$
such that $\pi\left(T_{v}\right)\not=0$ for all $v\in\Lambda_{0}$.
Then $\left(\pi\circ\rho_{\alpha}^{\alpha+1}\right)\left(S_{v}\right)\not=0$
for each $v\in\Lambda_{0}$. Since $\mathcal{O}\left(\Lambda_{\alpha}\right)$
has Cuntz-Krieger uniqueness by assumption, $\pi\circ\rho_{\alpha}^{\alpha+1}$
is injective, and in particular, $\left.\pi\right|_{\mathcal{A}_{\alpha}}$
is injective. Because $\psi_{\alpha}\times\pi_{\alpha}:\mathcal{O}\left(X_{\alpha}\right)\rightarrow\mathcal{O}\left(\Lambda_{\alpha+1}\right)$
is an isomorphism, $\left.\left(\pi\circ\psi_{\alpha}\times\pi_{\alpha}\right)\right|_{\eta_{\alpha}\left(\mathcal{A}_{\alpha}\right)}$
is injective. Applying \cite[Theorem 3.9]{CKU4CPA}, we see $\pi$
is injective.

Finally, we must show that if $\beta$ is a limit ordinal and for
all $\gamma<\beta$, $\mathcal{O}\left(\Lambda_{\gamma}\right)$ has
Cuntz-Krieger uniqueness, then $\mathcal{O}\left(\Lambda_{\beta}\right)$
has Cuntz-Krieger uniqueness. Let $\left\{ V_{e}:e\in\Lambda\right\} $
be the generators of $\mathcal{O}\left(\Lambda_{\beta}\right)$ and
$\pi:\mathcal{O}\left(\Lambda_{\beta}\right)\rightarrow\mathcal{C}$
be a {*}-homomorphism with $\pi\left(V_{v}\right)\not=0$ for all
$v\in\Lambda_{0}$. By \propref{inductive-system}, we have the following
\[
\mathcal{O}\left(\Lambda_{\beta}\right)=\overline{\bigcup_{\gamma<\beta}\rho_{\gamma}^{\beta}\left(\mathcal{O}\left(\Lambda_{\gamma}\right)\right)}
\]
For $\gamma<\beta$, $\pi\circ\rho_{\gamma}^{\beta}$ is injective,
hence $\ker\pi\cap\rho_{\gamma}^{\beta}\left(\mathcal{O}\left(\Lambda_{\gamma}\right)\right)=0$.
By \cite[II.8.2.4]{Encyclopaedia}, we get $\ker\pi=0$, and $\pi$
is injective.
\end{proof}

\section{Condition (S)}

Given an arbitrary ordinal graph $\Lambda$, it's not clear how to
verify condition (S) in \defref{condition-S}. Our final goal is to
derive a sufficient condition for condition (S) given an ordinal graph
which satsifies condition (V) defined below. We prove that it suffices
to check a class of directed graphs $\left\{ \mathcal{F}_{\alpha}:\alpha\in\mathrm{Ord}\right\} $.
Throughout this section, let $\Lambda$ be a fixed ordinal graph.
Recall that $e_{\alpha}$ and $e^{\alpha}$ defined in \defref{factor}
are the unique paths for which $d\left(e_{\alpha}\right)=\alpha$
and $e=e_{\alpha}e^{\alpha}$.
\begin{defn}
$\Lambda$ satisfies condition (V) if for every $\alpha\in\mathrm{Ord}$
and $e\in\Lambda^{\omega^{\alpha}}$, $e$ is $\alpha$-full.
\end{defn}
\begin{rem}
Since every path in $\Lambda^{\omega^{0}}=\Lambda^{1}$ is $0$-full,
every directed graph satisfies condition (V).
\end{rem}
\begin{defn}
For $\alpha\in\mathrm{Ord}$ and $f,g\in\Lambda^{\omega^{\alpha}}$,
let $f\cong_{\alpha}g$ if and only if there exist $\beta,\gamma<\omega^{\alpha}$
such that $f^{\beta}=g^{\gamma}$.
\end{defn}
Clearly the relation $\cong_{\alpha}$ is symmetric. It is also reflexive
since we may choose $\beta=\gamma=0$. Additionally, if $f,g,h\in\Lambda^{\omega^{\alpha}}$
with $f\cong_{\alpha}g$ and $g\cong_{\alpha}h$, choose $\beta,\gamma,\delta,\epsilon<\omega^{\alpha}$
such that $f^{\beta}=g^{\gamma}$ and $g^{\delta}=h^{\epsilon}$.
Then without loss of generality, assume $\gamma\leq\delta$, in which
case 
\[
f^{\beta-\gamma+\delta}=\left(f^{\beta}\right)^{-\gamma+\delta}=\left(g^{\gamma}\right)^{-\gamma+\delta}=g^{\gamma-\gamma+\delta}=g^{\delta}=h^{\epsilon}
\]
Hence $f\cong_{\alpha}h$. Thus $\cong_{\alpha}$ is an equivalence
relation.

Recall that for each ordinal graph $\Gamma$ we have an equivalence
relation $\sim$ from \defref{connected-components} such that for
$e,f\in\Gamma$, $e\sim f$ if and only if $e$ and $f$ belong to
the same connected component of $\Gamma$. In particular, we have
one such relation $\sim_{\alpha}$ for each $\alpha\in\mathrm{Ord}$
by setting $\Gamma=\Lambda_{\alpha}$.
\begin{defn}
For $\alpha\in\mathrm{Ord}$, define a directed graph $\mathcal{F}_{\alpha}=\left(\mathcal{F}_{\alpha}^{0},\mathcal{F}_{\alpha}^{1},r_{\alpha},s_{\alpha}\right)$
with vertices $\mathcal{F}_{\alpha}^{0}=\Lambda_{\alpha}/\sim_{\alpha}$
and edges $\mathcal{F}_{\alpha}^{1}=\Lambda^{\omega^{\alpha}}/\cong_{\alpha}$.
Define the range and source maps by
\begin{align*}
r_{\alpha}\left(\left[e\right]_{\cong_{\alpha}}\right) & =\left[r\left(e\right)\right]_{\sim_{\alpha}}\\
s_{\alpha}\left(\left[e\right]_{\cong_{\alpha}}\right) & =\left[s\left(e\right)\right]_{\sim_{\alpha}}
\end{align*}
\end{defn}
If $\beta,\gamma<\omega^{\alpha}$ with $e^{\beta}=f^{\gamma}$, then
$s\left(e\right)=s\left(e^{\beta}\right)=s\left(f^{\gamma}\right)=s\left(f\right)$.
Thus the source map $s_{\alpha}$ of $\mathcal{F}_{\alpha}$ is well-defined.
It remains to see the range map is well-defined. Suppose $e^{\beta}=f^{\gamma}$.
Then $d\left(e_{\beta}\right)=\beta<\omega^{\alpha}$, so $s\left(e_{\beta}\right)=r\left(e^{\beta}\right)\sim_{\alpha}r\left(e_{\beta}\right)=r\left(e\right)$.
For the same reason, $r\left(f^{\gamma}\right)\sim_{\alpha}r\left(f\right)$.
Since $r\left(e^{\beta}\right)=r\left(f^{\gamma}\right)$, we see
$r\left(e\right)\sim_{\alpha}r\left(f\right)$. Thus the range map
$r_{\alpha}$ is well-defined.

We analyze $\Lambda$ in terms of entries of cycles in the directed
graphs $\mathcal{F}_{\alpha}$. We use the definitions of cycle and
entry from \cite{GRAPHALGS}, which we record below.
\begin{defn}[{\cite[pg. 16]{GRAPHALGS}}]
If $E=\left(E^{0},E^{1},r_{E},s_{E}\right)$ is a directed graph,
a \emph{cycle} is a path $\mu=\mu_{1}\mu_{2}\ldots\mu_{n}$ with $n\geq1$,
$\mu_{k}\in E^{1}$, $s_{E}\left(\mu_{n}\right)=r_{E}\left(\mu_{1}\right)$
and $s_{E}\left(\mu_{j}\right)\not=s_{E}\left(\mu_{k}\right)$ for
$j\not=k$. An \emph{entry} to a cycle $\mu$ in $E$ is an edge $e\in E^{1}$
such that there is $j$ with $r_{E}\left(\mu_{j}\right)=r_{E}\left(e\right)$
and $\mu_{j}\not=e$.
\end{defn}
\begin{defn}[{\cite[Lemma 3.7]{GRAPHALGS}}]
If $E$ is a directed graph and $e=e_{1}e_{2}\ldots e_{n}$ is a
path in $E$ with $e_{j}\in E^{1}$, then $e$ is \emph{non-returning}
if $e_{j}\not=e_{n}$ for $j<n$.

Note that this definition of non-returning is slightly different than
the definition of a non-returning path in an ordinal graph from \defref{non-returning}.
In the following proof, we apply the above definition for the directed
graphs $\mathcal{F}_{\alpha}$, since it is more convenient, and use
\defref{non-returning} for ordinal graphs. Raeburn proves in \cite[Lemma 3.7]{GRAPHALGS}
that if every cycle of a directed graph has an entry, then for every
$n\in\mathbb{N}$ and $v\in E^{0}$ there exists a non-returning path
$e$ with $r\left(e\right)=v$ whose length is at least $n$. We use
this in the proof of the following theorem.
\end{defn}
\begin{thm}
\label{thm:condition-S}Suppose $\Lambda$ satisfies condition (V)
and for some $\alpha\in\mathrm{Ord}$, $\Lambda_{\alpha}$ satisfies
condition (S). If every cycle in $\mathcal{F}_{\alpha}$ has an entry,
then $\Lambda_{\alpha+1}$ satisfies condition (S).
\end{thm}
\begin{proof}
We must show for every $v\in\Lambda_{0}$ and $n\in\mathrm{Ord}$
with $n<\omega$, there exists $k$ with $\omega>k\geq n$ and $u\in\Lambda^{\omega^{\alpha}\cdot k}$
such that $r\left(u\right)=v$ and $u$ is non-returning. Since every
cycle of $\mathcal{F}_{\alpha}$ has an entry, choose $k\geq n$ and
a non-returning path $z=f_{1}f_{2}f_{3}\ldots f_{k}$ in $\mathcal{F}_{\alpha}$
with $r_{\alpha}\left(z\right)=\left[v\right]_{\sim_{\alpha}}$ and
$f_{j}\in\mathcal{F}_{\alpha}^{1}$. For each $1\leq j\leq k$, choose
$g_{j}\in\Lambda^{\omega^{\alpha}}$ such that $\left[g_{j}\right]_{\cong_{\alpha}}=f_{j}$.
Then $\left[r\left(g_{1}\right)\right]_{\sim_{\alpha}}=r_{\alpha}\left(f_{j}\right)=r_{\alpha}\left(z\right)=\left[v\right]_{\sim_{\alpha}}$,
hence $r\left(g_{1}\right)\sim_{\alpha}v$. Since $\Lambda$ satisfies
condition (V), there exists $\beta_{1}<\omega^{\alpha}$ and $h_{1}\in\Lambda_{\alpha}$
such that $s\left(h_{1}\right)=r\left(g_{1}^{\beta_{1}}\right)$ and
$r\left(h_{1}\right)=v$. Similarly, $r\left(g_{2}\right)\sim_{\alpha}s\left(g_{1}\right)$,
so we may choose $\beta_{2}<\omega^{\alpha}$ and $h_{2}\in\Lambda_{\alpha}$
such that $s\left(h_{2}\right)=r\left(g_{2}^{\beta_{2}}\right)$ and
$r\left(h_{2}\right)=s\left(g_{1}\right)$. Repeat this process for
$2<j\leq k$ to construct $h_{j}$ and $\beta_{j}<\omega^{\alpha}$
such that $s\left(h_{j}\right)=r\left(g_{j}^{\beta_{j}}\right)$ and
$r\left(h_{j}\right)=s\left(g_{j-1}\right)$. For each $j$ define
$p_{j}=h_{j}g_{j}^{\beta_{j}}$. 

Since $p_{j}^{d\left(h_{j}\right)}=\left(h_{j}g_{j}^{\beta_{j}}\right)^{d\left(h_{j}\right)}=g_{j}^{\beta_{j}}$,
$\left[p_{j}\right]_{\cong_{\alpha}}=f_{j}$. For $j<k$, $s\left(p_{j}\right)=s\left(g_{j}\right)=r\left(h_{j+1}\right)$.
Hence we may define $u=p_{1}p_{2}\ldots p_{k}$. Then $r\left(u\right)=r\left(p_{1}\right)=v$,
and because $d\left(h_{j}\right)<\omega^{\alpha}$ and $\beta_{j}<\omega^{\alpha}$,
\[
d\left(p_{j}\right)=d\left(h_{j}\right)+d\left(g_{j}^{\beta_{j}}\right)=d\left(h_{j}\right)-\beta_{j}+d\left(g_{j}\right)=d\left(h_{j}\right)-\beta_{j}+\omega^{\alpha}=\omega^{\alpha}
\]
Thus $d\left(u\right)=\omega^{\alpha}\cdot k$. We claim $u$ is non-returning
in the sense of \defref{non-returning}. Towards a contradiction,
suppose $u$ is not non-returning. Then there exists $e\in\Lambda$
with $\omega^{\alpha}\leq d\left(e\right)<\omega^{\alpha}\cdot k$
and $\beta<\omega^{\alpha}$ such that $eu^{\beta}\in u\Lambda$.
Let $q\in\Lambda$ with $eu^{\beta}=uq$. Choose $m<k$ such that
$\omega^{\alpha}\cdot m\leq d\left(e\right)<\omega^{\alpha}\cdot\left(m+1\right)$
and $\gamma<\omega^{\alpha}$ with $d\left(e\right)=\omega^{\alpha}\cdot m+\gamma$.
Choose $x_{1},x_{2},\ldots,x_{m}\in\Lambda^{\omega^{\alpha}}$ and
$t\in\Lambda_{\alpha}$ such that $e=x_{1}x_{2}\ldots x_{m}t$. Then
$u^{\beta}=\left(p_{1}\right)^{\beta}p_{2}\ldots p_{k}$, hence
\[
x_{1}x_{2}\ldots x_{m}t\left(p_{1}\right)^{\beta}p_{2}\ldots p_{k}=p_{1}p_{2}\ldots p_{k}q
\]
Now we split into two cases. In the first case, we have $k=m+1$.
Then by unique factorization,
\[
t\left(p_{1}\right)^{\beta}p_{2}\ldots p_{k}=p_{k}q
\]
Then $d\left(t\left(p_{1}\right)^{\beta}\right)=d\left(t\right)-\beta+d\left(p_{1}\right)=d\left(t\right)-\beta+\omega^{\alpha}=\omega^{\alpha}$.
Thus
\[
\left(p_{k}q\right)_{\omega^{\alpha}}=p_{k}=t\left(p_{1}\right)^{\beta}
\]
Therefore $\left(p_{k}\right)^{d\left(t\right)}=\left(p_{1}\right)^{\beta}$,
and $\left[p_{k}\right]_{\cong_{\alpha}}=f_{k}=\left[p_{1}\right]_{\cong_{\alpha}}=f_{1}$.
Since $z$ is a non-returning path in $\mathcal{F}_{\alpha}$, this
is a contradiction. For the other case, suppose $k>m+1$. Similarly,
unique factorization implies
\[
p_{k-m}p_{k-m+1}\ldots p_{k}=p_{k}q
\]
Thus $p_{k-m}=p_{k}$ and $f_{k-m}=f_{k}$, again contradicting the
fact that $z$ is non-returning.
\end{proof}
\begin{rem}
If $\beta$ is a limit ordinal and $\Lambda_{\alpha}$ satisfies condition
(S) for every $\alpha<\beta$, then $\Lambda_{\beta}$ satisfies condition
(S). Thus \thmref{condition-S} tells us that to verify $\Lambda$
satisfies condition (S), it suffices to check $\Lambda$ satisfies
condition (V) and for every $\alpha\in\mathrm{Ord}$, every cycle
of $\mathcal{F}_{\alpha}$ has an entry.
\end{rem}
\begin{example}
\label{exa:two-loops-two-omega}Consider the ordinal graph in \figref{two-loops-two-omega}
generated by one vertex, two paths of length 1, and two paths of length
$\omega$. $\mathcal{O}\left(\Lambda\right)$ is the $\mathrm{C}^{*}$-algebra
generated by two isometries $T_{e},T_{f}$ with $T_{e}T_{e}^{*}+T_{f}T_{f}^{*}=T_{v}=1$
and two isometries $T_{g},T_{h}$ with $T_{e}T_{g}=T_{g}$ and $T_{f}T_{h}=T_{h}$.
Then $\mathcal{F}_{0}$ and $\mathcal{F}_{1}$ each are directed graphs
with one vertex and two edges. $\Lambda$ satisfies condition (V)
since it has only one vertex, and every cycle in $\mathcal{F}_{0}$
and $\mathcal{F}_{1}$ has an entry. Therefore $\Lambda$ satisfies
condition (S). Because $\Lambda$ has no 1-sources and no 1-regular
vertices, \thmref{Cuntz-Krieger-Uniqueness-Theorem} implies $\mathcal{O}\left(\Lambda\right)$
has Cuntz-Krieger uniqueness. If $\pi:\mathcal{O}\left(\Lambda\right)\rightarrow\mathcal{C}$
is a {*}-homomorphism, then either $\pi$ is injective or $T_{v}\in\ker\pi$.
But $T_{v}$ is a unit for $\mathcal{O}\left(\Lambda\right)$, so
in this case $\ker\pi=\mathcal{O}\left(\Lambda\right)$. Thus $\mathcal{O}\left(\Lambda\right)$
is simple.
\begin{figure}[h]
\begin{center}
\begin{tikzpicture}
	\tikzset{every loop/.style={looseness=50}}
	\draw[->] node[circle, fill, inner sep=0, minimum size=4pt, label=above:$v$] (v) {} edge[in=115, out=180, loop] node[left] {$e$} ();
	\draw[->] (v) edge[in=65, out=0, loop] node[right] {$f$} ();
	\tikzset{every loop/.style={looseness=90}}
	\draw[->] (v) edge[in=200, out=250, loop] node[below, text width=2cm, align=center] {$g: eee\dots$ $g=eg$} ();
	\draw[->] (v) edge[in=340, out=290, loop] node[below, text width=2cm, align=center] {$h: fff\dots$ $h=fh$} ();
\end{tikzpicture}
\end{center}

\caption{\label{fig:two-loops-two-omega}The ordinal graph $\Lambda$ in \exaref{two-loops-two-omega}}
\end{figure}
\end{example}
We would like to end with a more complicated example. The previous
examples of ordinal graphs all have path lengths which are bounded
by $\omega^{2}$. Now we introduce an example of an ordinal graph
whose path lengths are bounded by $\omega^{\omega}$.
\begin{example}
\label{exa:complicated}Consider the directed graph in \figref{alpha-connected-component}.
We wish to construct an ordinal graph $\Lambda$ in which for every
$\alpha\in\mathbb{N}$, the connected components of $\mathcal{F}_{\alpha}$
are isomorphic to this graph. We consider an inductive construction,
starting with $\Gamma_{1}=\mathcal{F}_{0}$ as this directed graph.
Then to get $\Gamma_{2}$, we duplicate the previous directed graph
to get both vertices of $\mathcal{F}_{1}$. We then add a path of
length $\omega$ for each edge in $\mathcal{F}_{1}$. Duplicating
this ordinal graph again gives us both vertices of $\mathcal{F}_{2}$,
and adding paths of length $\omega^{2}$ gives us each edge in $\mathcal{F}_{2}$.
Regarding at each step $\Gamma_{k}$ as a subgraph of the ordinal
graph $\Gamma_{k+1}$ and taking the union $\cup_{k\in\mathbb{N}}\Gamma_{k}$,
we obtain an ordinal graph $\Lambda$ whose path lengths are bounded
by $\omega^{\omega}$. Of course, this construction is neither complete
nor formal, as we have not specified how these paths factor. Below
we formally define one choice for $\Lambda$ using generators and
relations.

Define $X=\left\{ x\in\left\{ 0,1\right\} ^{\mathbb{N}}:x^{-1}\left(1\right)\text{ is finite}\right\} $.
For $x\in X$, define $0x\in X$ and $1x\in X$ by
\begin{align*}
\left(0x\right)\left(n\right) & =\begin{cases}
x\left(n-1\right) & n\geq1\\
0 & n=0
\end{cases}\\
\left(1x\right)\left(n\right) & =\begin{cases}
x\left(n-1\right) & n\geq1\\
1 & n=0
\end{cases}
\end{align*}
Let $\Lambda$ be the category generated by $\left\{ v_{x},e_{x}^{\alpha},f_{x}^{\alpha},g_{x}^{\alpha},h_{x}^{\alpha}:x\in X,\alpha\in\mathbb{N}\right\} $
with the following relations:
\begin{enumerate}
\item $v_{0x}=s\left(e_{x}^{0}\right)=r\left(e_{x}^{0}\right)=s\left(g_{x}^{0}\right)=r\left(h_{x}^{0}\right)$
\item $v_{1x}=s\left(f_{x}^{0}\right)=r\left(f_{x}^{0}\right)=r\left(g_{x}^{0}\right)=s\left(h_{x}^{0}\right)$
\item $s\left(e_{x}^{\alpha+1}\right)=s\left(e_{0x}^{\alpha}\right)$
\item $s\left(f_{x}^{\alpha+1}\right)=s\left(f_{1x}^{\alpha}\right)$
\item $s\left(g_{x}^{\alpha+1}\right)=s\left(g_{0x}^{\alpha}\right)$
\item $s\left(h_{x}^{\alpha+1}\right)=s\left(h_{1x}^{\alpha}\right)$
\item $e_{0x}^{\alpha}e_{x}^{\alpha+1}=e_{x}^{\alpha+1}$
\item $f_{1x}^{\alpha}f_{x}^{\alpha+1}=f_{x}^{\alpha+1}$
\item $g_{1x}^{\alpha}h_{1x}^{\alpha}g_{x}^{\alpha+1}=g_{x}^{\alpha+1}$
\item $h_{0x}^{\alpha}g_{0x}^{\alpha}h_{x}^{\alpha+1}=h_{x}^{\alpha+1}$
\end{enumerate}
There is a length functor $d:\Lambda\rightarrow\mathrm{Ord}$ defined
by
\begin{align*}
d\left(v_{x}\right) & =0\\
d\left(e_{x}^{\alpha}\right) & =d\left(f_{x}^{\alpha}\right)=d\left(g_{x}^{\alpha}\right)=d\left(h_{x}^{\alpha}\right)=\omega^{\alpha}
\end{align*}
This makes $\Lambda$ an ordinal graph. Indeed, for each $\alpha\in\mathbb{N}$
the (infinitely many) connected components of $\mathcal{F}_{\alpha}$
are isomorphic to the graph in \figref{alpha-connected-component}.
$\Lambda$ satisfies condition (V), so by \thmref{condition-S}, $\Lambda$
satisfies condition (S).

\begin{figure}[h]
\begin{center}
\begin{tikzpicture}
	\tikzset{every loop/.style={looseness=50}}
	\draw (0, 0) node[circle, fill, inner sep=0, minimum size=4pt] (v) {};
	\draw (2, 0) node[circle, fill, inner sep=0, minimum size=4pt] (w) {};
	\draw[->] (v) edge[in=135, out=225, loop] node[left] {$e$} ();
	\draw[->] (w) edge[in=45, out=315, loop] node[right] {$f$} ();
	\draw[->] (v) edge[in=135, out=45] node[above] {$g$} (w);
	\draw[->] (w) edge[in=315, out=225] node[below] {$h$} (v);
\end{tikzpicture}
\end{center}

\caption{\label{fig:alpha-connected-component}A connected component of $\mathcal{F}_{\alpha}$
for $\Lambda$ defined in \exaref{complicated}}
\end{figure}
\end{example}
\pagebreak{}

\bibliographystyle{plain}
\bibliography{references/refs}

\begin{thebibliography}{10}

\bibitem{Encyclopaedia}
B.~Blackadar.
\newblock {\em Operator algebras}, volume 122 of {\em Encyclopaedia of
  Mathematical Sciences}.
\newblock Springer-Verlag, Berlin, 2006.
\newblock Theory of $C^*$-algebras and von Neumann algebras, Operator Algebras
  and Non-commutative Geometry, III.

\bibitem{CONDUCHEFIBRATIONS}
Jonathan~H. Brown and David~N. Yetter.
\newblock {Discrete Conduch\'e fibrations and $C^*$-algebras}.
\newblock {\em Rocky Mountain Journal of Mathematics}, 47(3):711 -- 756, 2017.

\bibitem{LYDIADEWOLF}
Lydia de~Wolf.
\newblock {\em Development of the theory of Kumjian-Pask fibrations, their path
  groupoids, and their C*-algebras}.
\newblock 2023.
\newblock PhD dissertation, Kansas State University.

\bibitem{CKU4CPA}
Menev\c{s}e Ery\"{u}zl\"{u} and Mark Tomforde.
\newblock A {C}untz-{K}rieger uniqueness theorem for {C}untz-{P}imsner
  algebras.
\newblock 2024.
\newblock (arXiv:2212.00248).

\bibitem{TOPALG}
Neal~J. Fowler and Iain Raeburn.
\newblock The {T}oeplitz algebra of a {H}ilbert bimodule.
\newblock {\em Indiana Univ. Math. J.}, 48(1):155--181, 1999.

\bibitem{KATIDEAL}
Takeshi Katsura.
\newblock A construction of {$C^*$}-algebras from {$C^*$}-correspondences.
\newblock In {\em Advances in quantum dynamics ({S}outh {H}adley, {MA}, 2002)},
  volume 335 of {\em Contemp. Math.}, pages 173--182. Amer. Math. Soc.,
  Providence, RI, 2003.

\bibitem{KATSURA1}
Takeshi Katsura.
\newblock A class of {$C^\ast$}-algebras generalizing both graph algebras and
  homeomorphism {$C^\ast$}-algebras. {I}. {F}undamental results.
\newblock {\em Trans. Amer. Math. Soc.}, 356(11):4287--4322, 2004.

\bibitem{KGRAPH}
Alex Kumjian and David Pask.
\newblock Higher rank graph {$C^\ast$}-algebras.
\newblock {\em New York J. Math.}, 6:1--20, 2000.

\bibitem{GRAPHALGS}
Iain Raeburn.
\newblock {\em Graph algebras}, volume 103 of {\em CBMS Regional Conference
  Series in Mathematics}.
\newblock Conference Board of the Mathematical Sciences, Washington, DC; by the
  American Mathematical Society, Providence, RI, 2005.

\bibitem{CARDINALORDINAL}
Wac{\l}aw Sierpi\'nski.
\newblock {\em Cardinal and ordinal numbers}, volume Vol. 34 of {\em Monografie
  Matematyczne [Mathematical Monographs]}.
\newblock Pa\'nstwowe Wydawnictwo Naukowe (PWN), Warsaw, revised edition, 1965.

\bibitem{LCSC}
Jack Spielberg.
\newblock Groupoids and {$C^*$}-algebras for left cancellative small
  categories.
\newblock {\em Indiana Univ. Math. J.}, 69(5):1579--1626, 2020.

\end{thebibliography}

\end{document}